\documentclass[10pt]{amsart}

\usepackage{amssymb}
\usepackage{amsthm}
\usepackage{amsmath}

\usepackage[usenames]{color}
\definecolor{ANDREW}{RGB}{255,127,0}

\usepackage[foot]{amsaddr}

\usepackage{hyperref}

\numberwithin{equation}{section}

\theoremstyle{plain}
\newtheorem{proposition}{Proposition}[section]
\newtheorem{theorem}[proposition]{Theorem}
\newtheorem{lemma}[proposition]{Lemma}
\newtheorem{corollary}[proposition]{Corollary}
\theoremstyle{definition}
\newtheorem{example}[proposition]{Example}
\newtheorem{definition}[proposition]{Definition}

\newtheorem{question}[proposition]{Question}
\theoremstyle{remark}
\newtheorem{remark}[proposition]{Remark}
\newtheorem{conjecture}[proposition]{Conjecture}

\DeclareMathOperator{\Ad}{Ad}
\DeclareMathOperator{\Aff}{Aff}
\DeclareMathOperator{\Aut}{Aut}

\DeclareMathOperator{\dimension}{dim}

\DeclareMathOperator{\tr}{tr}

\DeclareMathOperator{\SL}{SL}
\DeclareMathOperator{\GL}{GL}
\DeclareMathOperator{\Gr}{Gr}
\DeclareMathOperator{\SO}{SO}
\DeclareMathOperator{\PSO}{PSO}

\DeclareMathOperator{\PSL}{PSL}
\DeclareMathOperator{\PGL}{PGL}
\DeclareMathOperator{\Hom}{Hom}

\DeclareMathOperator{\End}{End}

\DeclareMathOperator{\Spanset}{Span} 
\DeclareMathOperator{\Proj}{Proj} 
\DeclareMathOperator{\Id}{Id} 
\DeclareMathOperator{\rank}{rank} 
\DeclareMathOperator{\length}{length} 
\DeclareMathOperator{\Isom}{Isom}
\DeclareMathOperator{\cd}{cd}
\DeclareMathOperator{\vcd}{vcd} 
\DeclareMathOperator{\Ext}{Ext_{\Rc}} 
\DeclareMathOperator{\Lsl}{\mathfrak{sl}} 
\DeclareMathOperator{\Lso}{\mathfrak{so}} 
\DeclareMathOperator{\Lin}{Lin} 
\DeclareMathOperator{\Inn}{Inn}

\DeclareMathOperator{\Bc}{\mathcal{B}}
\DeclareMathOperator{\Cc}{\mathcal{C}}

\DeclareMathOperator{\Fc}{\mathcal{F}}

\DeclareMathOperator{\Lc}{\mathcal{L}}
\DeclareMathOperator{\Nc}{\mathcal{N}}
\DeclareMathOperator{\Oc}{\mathcal{O}}

\DeclareMathOperator{\Rc}{\mathcal{R}}
\DeclareMathOperator{\Tc}{\mathcal{T}}

\DeclareMathOperator{\Cb}{\mathbb{C}}

\DeclareMathOperator{\Hb}{\mathbb{H}}

\DeclareMathOperator{\Kb}{\mathbb{K}}
\DeclareMathOperator{\Mb}{\mathbb{M}}
\DeclareMathOperator{\Nb}{\mathbb{N}}
\DeclareMathOperator{\Pb}{\mathbb{P}}
\DeclareMathOperator{\Rb}{\mathbb{R}}
\DeclareMathOperator{\Xb}{\mathbb{X}}
\DeclareMathOperator{\Zb}{\mathbb{Z}}

\let\inf\relax \DeclareMathOperator*\inf{\vphantom{p}inf}

\newcommand{\abs}[1]{\left|#1\right|}

\newcommand{\norm}[1]{\left\|#1\right\|}

\newcommand{\wh}[1]{\widehat{#1}}

\begin{document}

\title{Rigidity of convex divisible domains in flag manifolds}
\author{Wouter van Limbeek}
\email{wouterv@umich.edu}
\address{Department of Mathematics, University of Michigan, Ann Arbor, MI 48109.}
\author{Andrew Zimmer}
\email{aazimmer@uchicago.edu}
\address{Department of Mathematics, University of Chicago, Chicago, IL 60637.}\date{\today}
\subjclass[2010]{}

\maketitle

\begin{abstract} 
In contrast to the many examples of convex divisible domains in real projective space, we prove that up to projective isomorphism there is only one convex divisible domain in the Grassmannian of $p$-planes in $\Rb^{2p}$ when $p>1$. Moreover, this convex divisible domain is a model of the symmetric space associated to the simple Lie group $\SO(p,p)$. 
\end{abstract}

\tableofcontents 

\section{Introduction}
\label{sec:intro}

The Lie group $\PGL_{d+1}(\Rb)$ acts naturally on real projective space $\Pb(\Rb^{d+1})$ and for an open set $\Omega \subset \Pb(\Rb^{d+1})$ we define the \emph{automorphism group} of $\Omega$ as
\begin{align*}
\Aut(\Omega) = \{ \varphi \in \PGL_{d+1}(\Rb) : \varphi \Omega = \Omega \}.
\end{align*}
An open set $\Omega$ is then called a \emph{convex divisible domain} if it is a bounded convex open set in some affine chart of $\Pb(\Rb^{d+1})$ and there exists a discrete group $\Gamma \leq \Aut(\Omega)$ which acts properly, freely, and cocompactly on $\Omega$. 
 
The fundamental example of a convex divisible domain comes from the Klein-Beltrami model of real hyperbolic $d$-space $\Hb^d_{\Rb}$. In particular, if $\Bc \subset \Pb(\Rb^{d+1})$ is the unit ball in some affine chart, then $\Bc$ is a symmetric domain in the following sense: the group $\Aut(\Bc)$ acts transitively on $\Bc$, $\Aut(\Bc)$ is a simple group, and the stabilizer of a point $x \in \Bc$ is a maximal compact subgroup of $\Aut(\Bc)$. Moreover, there is a natural metric $H_{\Bc}$ on $\Bc$ called the Hilbert metric such that $(\Bc, H_{\Bc})$ is isometric to  $\Hb^d_{\Rb}$ and $\Aut(\Bc)$ coincides with $\Isom_0(\Hb^d_{\Rb})$. Finally since $\Aut(\Bc)$ is a simple Lie group, there exist cocompact torsion free lattices $\Gamma \leq \Aut(\Bc)$. Then any such $\Gamma$ acts properly, freely, and cocompactly on $\Bc$, so that $\Bc$ is a convex divisible domain.  

There are many more examples of convex divisible domains, for instance:
\begin{enumerate}
\item The symmetric spaces associated to $\SL_d(\Rb)$, $\SL_d(\Cb)$, $\SL_d(\Hb)$, and $E_{6(-26)}$ can all be realized as convex divisible domains.  For instance, consider the convex set
\begin{align*}
\mathcal{P} = \{ [X] \in \Pb(S_{d,d}) : X \text{ is positive definite}\}
\end{align*}
where $S_{d,d}$ is the vector space of real symmetric $d$-by-$d$ matrices. Then the group $\SL_{d}(\Rb)$ acts transitively on $\mathcal{P}$ by $g \cdot [X] = [g^t X g]$ and the stabilizer of a point is $\SO(d)$. Hence, if $\Gamma \leq \PSL_d(\Rb)$ is a cocompact torsion free lattice then 
$\Gamma$ acts properly, freely, and cocompactly on $\mathcal{P}$.
\item Let $\Bc\subseteq \Pb(\Rb^{d+1})$ be the Klein-Beltrami model of $\Hb^d_{\Rb}$. Results of Johnson-Millson~\cite{JM1987} and Koszul~\cite{K1968} imply that the domain $\Bc$ can be deformed to a divisible convex domain $\Omega$ where $\Aut(\Omega)$ is discrete (see~\cite[Section 1.3]{B2000} for $d>2$ and~\cite{G1990} for $d=2$).
\item Many examples in low dimensions (see for instance \cite{V1971, VK1967}). 
\item For every $d \geq 4$, Kapovich~\cite{K2007} has constructed divisible convex domains $\Omega \subset \Pb(\Rb^{d+1})$ such that $\Aut(\Omega)$ is discrete, Gromov hyperbolic, and not quasi-isometric to any symmetric space,.
\item Benoist~\cite{B2006} and Ballas, Danciger, and Lee~\cite{BDL2015} have constructed divisible convex domains $\Omega \subset \Pb(\Rb^{4})$ such that $\Aut(\Omega)$ is discrete, not Gromov hyperbolic, and not quasi-isometric to any symmetric space. 
\end{enumerate}
More background can be found in the survey papers by Benoist~\cite{B2008}, Goldman~\cite{Go2015}, Marquis~\cite{M2013}, and Quint~\cite{Q2010}.

In this paper we consider convex divisible domains in Grassmannians. The Lie group $\PGL_{p+q}(\Rb)$ acts naturally on the Grassmannian $\Gr_p(\Rb^{p+q})$ of $p$-planes in $\Rb^{p+q}$ and for an open set $\Omega \subset \Gr_p(\Rb^{p+q})$ we define the \emph{automorphism group} of $\Omega$ as
\begin{align*}
\Aut(\Omega) = \{ \varphi \in \PGL_{p+q}(\Rb) : \varphi \Omega = \Omega \}.
\end{align*}
As before, we say an open set $\Omega \subset \Gr_p(\Rb^{p+q})$ is a \emph{convex divisible domain} if it is a bounded convex open set in some affine chart of $\Gr_p(\Rb^{p+q})$ and there exists a discrete group $\Gamma \leq \Aut(\Omega)$ that acts properly, freely, and cocompactly on $\Omega$. 

As in the real projective setting, geometric models of symmetric spaces provide examples of convex divisible domains. The set of $q$-by-$p$ real matrices $M_{q,p}(\Rb)$ can be naturally identified with an affine chart of $\Gr_p(\Rb^{p+q})$. Now let $\Bc_{q,p}$ be the unit ball (with respect to the operator norm) in $M_{q,p}(\Rb)$. As in the real projective setting $\Bc_{q,p}$ is a symmetric domain: the group $\Aut(\Bc_{q,p})$ acts transitively on $\Bc_{q,p}$, $\Aut(\Bc_{q,p})$ is a simple group (in fact isomorphic to $\PSO(p,q)$), and the stabilizer of a point $x \in \Bc_{q,p}$ is a maximal compact subgroup in $\Aut(\Bc_{q,p})$. 

Given the plethora of convex divisible domains in real projective space, it is natural to ask:

\begin{question}
When $p,q > 1$, are there non-symmetric convex divisible domains in $\Gr_p(\Rb^{p+q})$? 
\end{question}

In contrast to the many examples of convex divisible domains in real projective space, we prove that every convex divisible domain in $\Gr_p(\Rb^{2p})$ is symmetric and even more precisely that up to projective isomorphism $\Bc_{p,p}$ is the only convex divisible domain in $\Gr_p(\Rb^{2p})$. The following is our main result.

\begin{theorem}\label{thm:main}
Suppose $p > 1$, $\Omega \subset \Gr_p(\Rb^{2p})$ is a bounded convex open subset of some affine chart, and there exists a discrete group $\Gamma \leq \Aut(\Omega)$ so that $\Gamma$ acts cocompactly on $\Omega$. Then $\Omega$ is projectively isomorphic to $\Bc_{p,p}$. 
\end{theorem}

\begin{remark}\label{rmk:assumptions} 
There is much more flexibility for domains which are not bounded in an affine chart. For instance, if $\Omega$ is an entire affine chart, there exists a discrete group $\Gamma \leq \Aut(\Omega)$ which acts freely, properly, and cocompactly on $\Omega$ (see Subsection~\ref{subsec:trivial_ex} below).

More generally there are recent constructions by Guichard-Wienhard~\cite{GW2008, GW2012}, Gu\'eritaud-Guichard-Kassel-Wienhard~\cite{GGKW2015}, and by Kapovich-Leeb-Porti~\cite{KLP2013, KLP2014b, KLP2014} of open domains $\Omega$ in certain flag manifolds where there exists a discrete group $\Gamma \leq \Aut(\Omega)$ that acts properly, freely, and cocompactly on $\Omega$. These constructions come from the theory of Anosov representations, and give many examples of nonsymmetric divisible domains $\Omega$. However, (to the best of our knowledge) these constructions never produce domains which are bounded in any affine chart. 
\end{remark}

\begin{remark} It is well-known that convex domains in real projective space are very similar to nonpositively curved Riemannian manifolds (see for instance~\cite{B2004, B2006, C2009, CLT2015}). In particular the flexibility of domains in real projective space and the rigidity of domains in $\Gr_p(\Rb^{2p})$ when $p > 1$ can be compared to the well known dichotomy for the rigidity of a nonpositively curved metric based on its \emph{Euclidean rank}. Nonpositively curved metrics of rank one are very flexibile (e.g. negatively curved metrics), but in higher rank there is an amazing amount of rigidity. Namely, the Higher Rank Rigidity Theorem of Ballmann~\cite{Ba85} and Burns-Spatzier~\cite{BS87a, BS87b} states that any nonpositively curved, irreducible, closed Riemannian manifold whose Euclidean rank is at least two, is isometric to a locally symmetric space. 

Inspired by this analogy we conjecture that a version of Theorem~\ref{thm:main} also holds for $p\neq q$, as long as $p, q>1$ (see Conjecture \ref{conj:pnotq} below for a precise statement). However, our methods do not extend to this setting (see Remark \ref{rmk:pnotqfail} below for more information). \end{remark}

\begin{remark} In Theorem \ref{thm:main} we only assume that there is a discrete group $\Gamma\leq \Aut(\Omega)$ acting cocompactly on $\Omega$. However this implies that there exists a discrete group $\Gamma_0 \leq \Aut(\Omega)$ which acts acts freely, properly discontinuously, and cocompactly on $\Omega$. Namely, by Proposition~\ref{prop:proper} below, $\Aut(\Omega)$ acts properly on $\Omega$. Thus if $\Gamma \leq \Aut(\Omega)$ is a discrete group and $\Gamma$ acts cocompactly on $\Omega$ then $\Gamma$ is finitely generated (by the \v{S}varc-Milnor lemma, see~\cite[Chapter I.8 Proposition 8.19]{BH1999}). Then Selberg's lemma (see~\cite{A1987}) implies that $\Gamma$ has a finite index torsion free subgroup $\Gamma_0 \leq \Gamma$. Then $\Gamma_0$ acts freely, properly discontinuously, and cocompactly on $\Omega$. 
\end{remark}

\subsection{General flag manifolds} Suppose $1 \leq d_1 < \dots < d_r \leq d$ and $\Kb$ is either the real numbers $\Rb$, the complex numbers $\Cb$, or the quaternions $\Hb$. Then let $\Fc(d_1, \dots, d_r; \Kb^d)$ denote the space of flags
\begin{align*}
(0) \leq V_1 \leq \dots \leq V_r \leq \Kb^d
\end{align*}
with $\dim_{\Kb} V_i = d_i$. The group $\PGL_d(\Kb)$ acts on $\Fc(d_1, \dots, d_r;\Kb^d)$ and for an open set $\Omega \subset \Fc(d_1, \dots, d_r;\Kb^d)$ we define the \emph{automorphism group} of $\Omega$ as
\begin{align*}
\Aut(\Omega) = \{ \varphi \in \PGL_{d}(\Rb) : \varphi \Omega = \Omega \}.
\end{align*}
As before, we say an open set $\Omega \subset  \Fc(d_1, \dots, d_r;\Kb^d)$ is a \emph{convex divisible domain} if it is a bounded convex open set in some affine chart of  $\Fc(d_1, \dots, d_r;\Kb^d)$ and there exists a discrete group $\Gamma \leq \Aut(\Omega)$ which acts properly, freely, and cocompactly on $\Omega$. 

\begin{question}
Are there non-symmetric convex divisible domains in $\Fc(d_1, \dots, d_r; \Kb^d)$? 
\end{question}

There are a number of partial answers:
\begin{enumerate}
\item When $r > 1$, there are no convex divisible domains in $\Fc(d_1, \dots, d_r; \Rb^d)$ by~\cite[Theorem 1.11]{Z2015}. The same argument can be used to show that there are no convex divisible domains in $\Fc(d_1, \dots, d_r; \Kb^d)$ for $\Kb = \Cb$ or $\Kb=\Hb$.
\item When $\Kb=\Cb$, a result of Frankel~\cite{F1989} implies that every convex divisible domain in $\Gr_p(\Cb^{p+q})$ is a bounded symmetric domain in the sense that $\Aut(\Omega)$ is a semisimple Lie group which acts transitively on $\Omega$. 
\item When $\Kb = \Hb$, an argument of Frankel~\cite[Section 6]{F1989} can be used to show that $\Aut(\Omega)$ is non-discrete. 
\end{enumerate}

These partial answers motivate the following:

\begin{conjecture}\label{conj:pnotq}
If $\Omega \subset \Gr_p(\Kb^{p+q})$ is a convex divisible domain and $p,q > 1$, then $\Omega$ is a bounded symmetric domain.
\end{conjecture}

\subsection{Outline of the proof of Theorem~\ref{thm:main}} The proof of Theorem~\ref{thm:main} uses a variety of techniques from real projective geometry, several complex variables, Riemannian geometry, Lie theory, and algebraic topology. Here is an outline of the three mains steps: \newline

\noindent \textbf{Step 1: Constructing an invariant metric}.  A convex domain $\Omega$ in an affine chart of $\Pb(\Rb^{d+1})$ that is proper 
(that is, does not contain any affine real lines) has a complete metric called the Hilbert metric. 
One of the main steps in the proof is the construction of a metric $K_\Omega$ that generalizes this classical construction.

We say a convex domain $\Omega$ in an affine chart of $\Gr_p(\Rb^{p+q})$ is $\Rc$-proper if it does contain any ``rank one affine real lines'' (see Definition~\ref{defn:r_proper} below). 

\begin{theorem}\label{thm:intrinsic_metric}(Theorem~\ref{thm:completeness} and Theorem~\ref{thm:haus_conv} below)
Suppose $\Mb \subset \Gr_p(\Rb^{p+q})$ is an affine chart and $\Omega \subset \Mb$ is a $\Rc$-proper convex subset of $\Mb$. Then there exists a complete  length metric $K_{\Omega}$ with the following properties:
\begin{enumerate}
\item (Invariance) the group $\Aut(\Omega)$ acts by isometries on $(\Omega, K_\Omega)$. 
\item (Equivariance) if $\Phi \in  \PGL_{p+q}(\Rb)$ then 
\begin{align*}
K_{\Omega}(x,y) = K_{\Phi\Omega}(\Phi x, \Phi y),
\end{align*} 
\item (Continuity in the local Hausdorff topology) if $\Omega_n \subset \Mb$ is a sequence of $\Rc$-proper convex sets converging in the local Hausdorff topology to a $\Rc$-proper convex open set $\Omega \subset \Mb$ then $K_{\Omega_n}$ converges to $K_\Omega$, 
\item if $p=1$ then $K_\Omega$ coincides with the classical Hilbert metric. 
\end{enumerate}
\end{theorem}

The above theorem allow us to establish an analogue of the powerful ``rescaling'' method from several complex variables (see the survey articles~\cite{F1991, KK2008}), although for completely different reasons. See Remark \ref{rmk:rescaling} below for further details on this analogy ( or lack thereof). We prove:

\begin{theorem}\label{thm:rescaling_intro}(Theorem~\ref{thm:rescaling} below)
Suppose $\Mb \subset \Gr_p(\Rb^{p+q})$ is an affine chart, $\Omega \subset \Mb$ is an $\Rc$-proper convex subset of $\Mb$, and $\Aut(\Omega)$ acts cocompactly on $\Omega$. If $A_n \in \Aff(\Mb) \cap \PGL_{p+q}(\Rb)$  and $A_n \Omega$ is a sequence of $\Rc$-proper convex sets converging in the local Hausdorff topology to an $\Rc$-proper convex open set $\wh{\Omega}$ then there exists some $\Phi \in \PGL_{p+q}(\Rb)$ so that $\Phi(\Omega) = \wh{\Omega}$.
\end{theorem}

To explain how the properties of the metric $K_\Omega$ imply Theorem~\ref{thm:rescaling_intro}, let us sketch the proof:

\begin{proof}[Proof Sketch] 
Suppose that $A_n \Omega \rightarrow \wh{\Omega}$. Fix a point $x_0 \in \Omega$. Since $\Aut(\Omega)$ acts cocompactly on $\Omega$, we can pass to a subsequence and find $\varphi_n \in \Aut(\Omega)$ so that $A_n\varphi_n x_0 \rightarrow \wh{x}_0 \in \wh{\Omega}$. Now consider the maps $f_n: = A_n \varphi_n$. By part (1) and (2) of Theorem~\ref{thm:intrinsic_metric}, each $f_n$ induces an isometry $(\Omega, K_\Omega) \rightarrow (\Omega_n, K_{\Omega_n})$.  Then by part (3) of Theorem~\ref{thm:intrinsic_metric}, one can pass to a subsequence so that $f_n \rightarrow f$ and $f$ will be an isometry $(\Omega, K_\Omega) \rightarrow (\wh{\Omega}, K_{\wh{\Omega}})$. A simple argument then shows that $f$ is actually the restriction of a element in $\PGL_{p+q}(\Rb)$. 
\end{proof}

Theorem~\ref{thm:rescaling_intro} should also be compared to a theorem of Benz{\'e}cri from real projective geometry. Let $\Xb_d$ be the space of proper convex open sets in $\Pb(\Rb^d)$ with the Hausdorff topology. Then $\Xb_d$ is closed in the Hausdorff topology and $\PGL_d(\Rb)$ acts on $\Xb_d$. With this notation Benz{\'e}cri proved:

\begin{theorem}\cite{B1960}(see also~\cite[Theorem 6.15]{G2009})
Suppose $\Omega$ is a proper convex open set in $\Pb(\Rb^d)$. If $\Aut(\Omega)$ acts cocompactly on $\Omega$ then $\PGL_d(\Rb) \cdot \Omega$ is a closed subset of $\Xb_d$. 
\end{theorem}

It is important to note that unlike in the real projective setting, when $p,q > 1$, convexity is not invariant under the action of $\PGL_{p+q}(\Rb)$ on $\Gr_p(\Rb^{p+q})$: If $\Omega$ is a convex subset of some affine chart $\Mb \subset \Gr_p(\Rb^{p+q})$ and $\phi(\Omega) \subset \Mb$ for some $\phi \in \PGL_p(\Rb^{p+q})$, then $\phi(\Omega)$ may not be a convex subset of $\Mb$. Thus to preserve convexity we are forced to consider the orbit of $\Omega$ under the group  $\Aff(\Mb) \cap \PGL_{p+q}(\Rb)$. \newline

\noindent \textbf{Step 2: The automorphism group is non-discrete}. In the second step of the proof we use the rescaling theorem from step one to show that $\Aut(\Omega)$ is non-discrete when $\Omega \subset \Gr_p(\Rb^{2p})$ is a convex divisible domain. 

We can identify $M_{p,p}(\Rb)$ with the affine chart 
\begin{align*}
\left\{ \begin{bmatrix} \Id_p \\ X \end{bmatrix} : X \in M_{p,p}(\Rb) \right\}
\end{align*}
of $\Gr_p(\Rb^{2p})$. Then consider the symmetric domain $\Bc_{p,p} \subset M_{p,p}(\Rb)$. Note that $\Bc_{p,p}$ is a convex set and the extreme points of $\Bc_{p,p}$ are exactly the orthogonal matrices. Given an orthogonal matrix $A \in \partial \Bc_{p,p}$, define the projective transformation 
\begin{align*}
F(X) := \begin{bmatrix} -\Id_p & A^{-1} \\ \Id_p & A^{-1} \end{bmatrix} \cdot X =  (A^{-1}X+\Id_p)(A^{-1}X-\Id_p)^{-1}.
\end{align*}
Then we see that 
\begin{align*}
F(\Bc_{p,p}) = \{ X \in M_{p,p}(\Rb): X^t + X > 0\}
\end{align*}
and $F(A) = 0$. Now $F(\Bc_{p,p})$ is a cone and in particular $\Aut(F(\Bc_{p,p}))$ contains a one-parameter group of homotheties. Translating this back to $\Bc_{p,p}$ shows that $A \in \partial \Bc_{p,p}$ is the attracting fixed point of a one-parameter group of automorphisms of $\Bc_{p,p}$. 

Using the rescaling theorem from Step 1 we will recover these one parameter groups for a general divisible domain. The key result is the following: 

\begin{theorem}\label{thm:tangent_cone}(Theorem~\ref{thm:extreme} below)
Suppose $\Mb \subset \Gr_p(\Rb^{2p})$ is an affine chart, $\Omega \subset \Mb$ is a $\Rc$-proper convex subset of $\Mb$, and $\Aut(\Omega)$ acts cocompactly on $\Omega$. If $e \in \partial \Omega$ is an extreme point, then the tangent cone of $\Omega$ at $e$ is $\Rc$-proper. 
\end{theorem}

Now the tangent cone of $\Omega$ at $e$ is precisely the limit of the rescaled domains
\begin{align*}
n (\Omega - e) + e
\end{align*}
in the local Hausdorff topology. In particular combining Theorem~\ref{thm:rescaling_intro} and Theorem~\ref{thm:tangent_cone} implies the following:

\begin{corollary}(Corollary~\ref{cor:rescaling} below)
Suppose $\Mb \subset \Gr_p(\Rb^{2p})$ is an affine chart, $\Omega \subset \Mb$ is a $\Rc$-proper convex subset of $\Mb$, and $\Aut(\Omega)$ acts cocompactly on $\Omega$. Then $\Aut(\Omega)$ is non-discrete. 
\end{corollary}

\begin{remark}\label{rmk:rescaling}
 In the several complex variable setting, rescaling can also be used to find one-parameter groups of automorphisms (see~\cite[Section 6]{F1989} or~\cite{K2004}). However, in this setting one obtains these automorphisms
by rescaling at a point in the boundary with either $C^1$ or $C^2$ regularity. 
This procedure actually finds automorphisms because a complex line has two real dimensions (see the proof of~\cite[Lemma 6.8]{F1991}).
In contrast we find a one-parameter group of automorphisms by rescaling at a point where the tangent cone is $\Rc$-proper and hence 
very far from being $C^1$. 
Finally, we should observe that the rescaling method cannot be used to find one parameter groups of automorphisms in the real projective setting.
\end{remark}

\begin{remark}\label{rmk:pnotqfail}
If $p\neq q$, an explicit computation for $\Bc_{p,q}$ shows that Theorem \ref{thm:tangent_cone} fails in this setting. This is one of the main problems that prevent us from extending our methods to the general case.
\end{remark}

\noindent \textbf{Step 3: Showing the automorphism group is simple and acts transitively} In the final part of the proof we show that $\Aut_0(\Omega)$, the connected component of the identity of $\Aut(\Omega)$, is a simple Lie group which acts transitively on $\Omega$. 

Our approach for this step is based on work of Farb and Weinberger~\cite{FW2008} who prove a number of remarkable rigidity results for compact aspherical Riemannian manifolds whose universal covers have non-discrete isometry groups. In particular, we combine their approach with the representation theory of Lie groups to establish the following: 

\begin{theorem}(see Theorem~\ref{thm:reduction} below)\label{thm:reduction_intro} Suppose $p > 1$, $\Mb \subset \Gr_p(\Rb^{2p})$ is an affine chart, $\Omega \subset \Mb$ is a bounded convex open subset of $\Mb$, and there exists a discrete group $\Gamma \leq \Aut(\Omega)$ so that $\Gamma \backslash \Omega$ is compact. Then at least one of the following holds:
\begin{enumerate}
\item a finite index subgroup of $\Gamma$ has non-trivial centralizer in $\PGL_{2p}(\Rb)$, 
\item there exists a nontrivial abelian normal unipotent group $U \leq \Aut(\Omega)$ such that $\Gamma \cap U$ is a cocompact lattice in $U$, 
\item $p=2$ and there exists a finite index subgroup $G'$ of $\Aut(\Omega)$ such that $G'=\Aut_0(\Omega)\times \Lambda$ for some discrete group $\Lambda$. Further up to conjugation
\begin{align*}
\Aut_0(\Omega) = \left\{ \begin{bmatrix} A & 0 \\ 0 & A \end{bmatrix} : A \in \SL_2(\Rb) \right\}
\end{align*}
and 
\begin{align*}
\Lambda \leq \left\{ \begin{bmatrix} a \Id_2 & b \Id_2 \\ c\Id_2 & d\Id_2 \end{bmatrix} : ad-bc = 1\right\}.
\end{align*}
\item $p=2$, $\Aut_0(\Omega) \leq \Aut(\Omega)$ has finite index and acts transitively on $\Omega$, and up to conjugation
\begin{align*}
\Aut_0(\Omega)=\left\{ \begin{bmatrix} aA & bA \\ cA & dA \end{bmatrix} : A \in \SL_2(\Rb), ad-bc = 1\right\}.
\end{align*}
\item $\Aut_0(\Omega)$ is a simple Lie group with trivial center that acts transitively on $\Omega$.
\end{enumerate}
\end{theorem}

In Sections~\ref{sec:center},~\ref{sec:unipotent}, and~\ref{sec:p_2} we use the dynamics of the action of $\PGL_{2p}(\Rb)$ on $\Gr_p(\Rb^{2p})$ to show that the first four cases in Theorem~\ref{thm:reduction_intro} are impossible. Finally in Section~\ref{sec:final_step} we use the classification of simple Lie groups and the representation theory of simple Lie groups to complete the proof Theorem~\ref{thm:main}.

\subsection*{Acknowledgments} 

We would like to thank Benson Farb and Ralf Spatzier for many helpful conversations. The first author gratefully acknowledges support from the University of Chicago while part of this work was done. The second author was partially supported by the National Science Foundation under grant number 1045119 and grant number 1400919.

\section{Preliminaries}
\label{sec:prelim}

\subsection{Notations} Given some object $o$ we will let $[o]$ be the projective equivalence class of $o$, for instance: if $v \in \Rb^{d+1} \setminus \{0\}$ let $[v]$ denote the image of $v$ in $\Pb(\Rb^{d+1})$; if $\phi \in \GL_{d+1}(\Rb)$ let $[\phi]$ denote the image of $\phi$ in $\PGL_{d+1}(\Rb)$; if $T \in \Lin(\Rb^{d_1+1}, \Rb^{d_2+1}) \setminus\{0\}$ let $[T]$ denote the image of $T$ in $\Pb(\Lin(\Rb^{d_1+1}, \Rb^{d_2+1}))$. 

\subsection{The Hilbert metric}\label{subsec:hilbert} The Hilbert metric is classically only defined for convex domains in real projective space, but Kobayashi~\cite{K1977} gave a construction that works for any open connected domain in real projective space. In this subsection we recall Kobayashi's construction. 

Given four points $a, x,y, b \in \Pb(\Rb^{d})$ that are collinear, that is contained in a projective line, one can define the cross ratio by 
\begin{align*}
[a;x;y;b] = \log \frac{ \abs{x-b}\abs{y-a}}{\abs{x-a}\abs{y-b}}.
\end{align*}
The cross ratio is $\PGL_d(\Rb)$-invariant in the sense that 
\begin{align*}
[a;x;y;b] = [\varphi a;\varphi x;\varphi y;\varphi b]
\end{align*}
for any $\varphi \in \PGL_d(\Rb)$. 

Next let 
\begin{align*}
I:= \{ [1:t] \in \Pb(\Rb^2) : \abs{t} < 1\}
\end{align*}
be the unit interval and consider the function $H_I:I\times I \rightarrow \Rb_{\geq 0}$ given by
\begin{align*}
H_I(s,t) = \abs{ \log [-1;s;t;1] }.
\end{align*}
Then $H_I$ is a  complete $\Aut(I)$-invariant length metric on $I$. 

Now suppose that $\Omega \subset \Pb(\Rb^d)$ is an open connected set. Let 
\begin{align*}
\Proj(I, \Omega) \subset \Pb( \End(\Rb^2, \Rb^d))
\end{align*}
 be the set of projective maps $T$ so that $I \cap \ker T = \emptyset$ and $T(I) \subset \Omega$. Then define a function $\rho_\Omega : \Omega \times \Omega \rightarrow \Rb \cup \{ \infty\}$ as follows: 
\begin{align*}
\rho_\Omega(x,y) := \inf\left\{ H_I(s,t) : \text{ there exists } f \in \Proj(I, \Omega) \text{ with } f(s) = x \text{ and } f(t) = y \right\}.
\end{align*}
Finally, using $\rho_{\Omega}$, one defines the \emph{pseudo-metric} $K_{\Omega}$ as
\begin{align*}
K_\Omega(x,y) = \inf\left\{ \sum_{i=0}^{N-1} \rho_\Omega(x_i,x_{i+1}) : N>0, x_0, \dots, x_N \in \Omega, x_0 = x, x_N = y\right\}.
\end{align*}

With this definition, Kobayashi proved the following:

\begin{theorem}\cite{K1977} Suppose $\Omega \subset \Pb(\Rb^d)$ is a open connected set. Then 
\begin{enumerate}
\item $K_\Omega$ is a $\Aut(\Omega)$-invariant pseudo-metric on $\Omega$, 
\item if $\Omega$ is bounded in an affine chart then $K_{\Omega}$ is a metric, 
\item if $\Omega$ is convex then $K_{\Omega}$ coincides with the Hilbert metric, and
\item $K_{\Omega}$ is a complete metric if and only if $\Omega$ is convex. 
\end{enumerate}
\end{theorem}

\section{The Grassmannians}

In this expository section we recall the two standard models of the Grassmannians, define affine charts, and describe the projective lines contained in the Grassmannians. 

\subsection{The matrix model} 
We can identify $\Gr_p(\Rb^{p+q})$ with the quotient 
\begin{align*}
\{ X \in M_{p+q, p}(\Rb) : \rank X = p \} / \GL_{p}(\Rb)
\end{align*}
where $X \mapsto \operatorname{Im}(X)$. 

\subsection{The projective model}\label{subsec:embedding}

We have a natural embedding $\Gr_p(\Rb^{p+q}) \rightarrow \Pb(\wedge^p \Rb^{p+q})$ defined by
\begin{align*}
\Spanset(v_1, \dots, v_p) \rightarrow [v_1 \wedge \dots \wedge v_p].
\end{align*}
This is well-defined and the image is a closed algebraic set in $\Pb(\wedge^p \Rb^{p+q})$.

\subsection{Affine charts}\label{subsec:affine_charts}

Suppose $W_0$ is a $q$-dimensional subspace of $\Rb^{p+q}$. Then consider the set 
 \begin{align*}
 \Mb := \{ U \in \Gr_p(\Rb^{p+q}) : U \cap W_0 = (0)\}.
 \end{align*}
 Note that $\Mb$ is an open dense subset of $\Gr_p(\Rb^{p+q})$. We call $\Mb$ an affine chart. 

 If we fix a subspace $U_0 \in \Mb$, we can identify $\Mb$ with the set $\Hom(U_0, W_0)$ via 
\begin{align*}
\Hom(U_0, W_0) &\longrightarrow \Mb\\
T&\mapsto \operatorname{Graph}(T): = \{ (\Id+T)u : u \in U_0\}.
\end{align*} 
Fixing bases of $U_0$ and $W_0$ gives an identification of $\Mb$ with the space of $q$-by-$p$ real matrices. Notice that a different choice of bases or of $U_0$ only changes this identification by a map of the form 
\begin{align*}
X \mapsto A X B+C
\end{align*}
where $A \in \GL_q(\Rb)$, $B \in \GL_p(\Rb)$, and $C$ is a $q$-by-$p$ matrix. This observation leads to the next definition:

\begin{definition}
For an affine chart $\Mb \subset \Gr_p(\Rb^{p+q})$ let $\Aff(\Mb)$ be the transformations of $\Mb$ that are affine maps with respect to some (hence any) identification of $\Mb$ with the space of $q$-by-$p$ real matrices.
\end{definition}

If $\Mb$ is an affine chart then there exists $g \in \PGL_{p+q}(\Rb)$ so that 
\begin{align*}
g \Mb = \left\{ \begin{bmatrix} \Id_p \\ X \end{bmatrix} : X \in M_{q,p}(\Rb) \right\} 
\end{align*}
in the matrix model. Moreover, if $e_1, \dots, e_{p+q}$ is the standard basis of $\Rb^{p+q}$ then 
\begin{align*}
g\Mb = \{ \left[ (e_1+v_1) \wedge \dots \wedge (e_p+v_p) \right]: v_1, \dots, v_p \in \Spanset\{e_{p+1}, \dots, e_{p+q} \}\}
\end{align*}
in the projective model. 

\subsection{Projective lines in the two models }

\begin{lemma}
If $\ell$ is a projective line in $\Pb(\wedge^p \Rb^{p+q})$ contained in $\Gr_p(\Rb^{p+q})$ then there exists $v_1, \dots, v_p, w \in \Rb^{p+q}$ so that 
\begin{align*}
\ell = \Big\{ [v_1 \wedge \dots \wedge v_{p-1} \wedge (v_p+t w)]  : t \in \Rb\Big\} \cup \Big\{ [v_1 \wedge \dots \wedge v_{p-1} \wedge w]\Big\}.
\end{align*}
\end{lemma}

\begin{proof}
Recall that for an element $x \in \wedge^p \Rb^{p+q}$, we have that $[x]$ belongs to $\Gr_p(\Rb^{p+q})$ if and only if the linear map $T_x : \Rb^{p+q} \rightarrow \wedge^{p+1} \Rb^{p+q}$ given by $T_x(v) = v \wedge x$ has rank $q$.

Now since $\ell$ is a projective line there exist $w_1, \dots, w_p, v_1, \dots, v_p \in \Rb^{p+q}$ so that 
\begin{align*}
\ell = \Big\{ [(v_1 \wedge \dots \wedge v_p) + t( w_1 \wedge \dots \wedge w_p) ]: t \in \Rb\Big\} \cup \Big\{  [w_1 \wedge \dots \wedge w_p]\Big\}.
\end{align*}
Let
\begin{align*}
V = \Spanset\{ v_1, \dots, v_p\} \cap \Spanset\{ w_1, \dots, w_p\}
\end{align*}
and $r = \dimension V$. We claim that $r=p-1$.

We can assume that $v_i = w_i$ for $1 \leq i \leq r$ and thus $v_1, \dots, v_p, w_{r+1}, \dots, w_p$ are all linearly independent.  So if 
\begin{align*}
x_t = (v_1 \wedge \dots \wedge v_p) + t( w_1 \wedge \dots \wedge w_p)
\end{align*}
and $v \wedge x_t =0$ then either $v \in V$ or 
\begin{align*}
v \wedge v_1 \wedge \dots \wedge v_p =  -t ( v \wedge w_1 \wedge \dots \wedge w_p) \neq 0.
\end{align*}
This last case is only possible when $r = p-1$ and $v= v_p -t w_p$.  Since $\dimension \ker T_{x_t} = p$ and $\dimension V = r \leq p-1$ this implies that $r=p-1$. Then 
\begin{align*}
[(v_1 \wedge \dots \wedge v_p) + t( w_1 \wedge \dots \wedge w_p) ] = [ v_1 \wedge \dots \wedge v_{p-1} \wedge (v_p + t w_p) ]
\end{align*}
for all $t \in \Rb$, which implies the lemma. 
\end{proof}

\begin{corollary}
Suppose $x,y \in \Gr_p(\Rb^{p+q})$. Then the following are equivalent:
\begin{enumerate}
\item There exists a  projective line $\ell$ in $\Pb(\wedge^p \Rb^{p+q})$ contained in $\Gr_p(\Rb^{p+q})$ so that $x,y \in \ell$, 
\item $\dimension(x \cap y) \geq p-1$.
\end{enumerate}
\end{corollary}

\begin{lemma}
\label{lem:rank_one_lines}
Suppose $\Mb$ is an affine chart in $\Gr_p(\Rb^{p+q})$ and we identify $\Mb$ with the set of $q$-by-$p$ matrices. Then
\begin{enumerate}
\item if $\ell$ is a projective line in $\Pb(\wedge^p \Rb^{p+q})$ contained in $\Gr_p(\Rb^{p+q})$ and $\ell \cap \Mb \neq \emptyset$, then
\begin{align*}
\ell \cap \Mb = \{ X+tS : t \in \Rb\}
\end{align*}
for some $X,S \in \Mb$ with $\rank(S)=1$. 
\item Conversely, if $X,S \in \Mb$ and $\rank(S)=1$ then the closure of 
\begin{align*}
\{ X+tS : t \in \Rb\}
\end{align*}
in $\Pb(\wedge^p \Rb^{p+q})$ is a projective line. 
\end{enumerate}
\end{lemma}

\begin{proof}
First suppose that $\ell$ is a projective line contained in $\Gr_p(\Rb^{p+q})$ and $\ell \cap \Mb \neq \emptyset$. There exists some $W_0 \in \Gr_q(\Rb^{p+q})$ such that $\Mb = \{ U \in \Gr_p(V) : U \cap W_0 = (0) \}$. By the previous Lemma we can assume 
\begin{align*}
\ell = \Big\{ [v_1 \wedge \dots \wedge v_{p-1} \wedge (v_p+t w)]  : t \in \Rb\Big\} \cup \Big\{ [v_1 \wedge \dots \wedge v_{p-1} \wedge w]\Big\}.
\end{align*}
for some $w,v_1, \dots, v_p \in \Rb^{p+q}$. By modifying these vectors we can assume that $[v_1 \wedge \dots \wedge v_p] \in \Mb$ and $w \in W_0$ (in particular $[w \wedge v_2 \wedge \dots \wedge v_p] \notin \Mb$). Let $U_0 = \Spanset\{ v_1, \dots, v_p\}$ and identify $\Mb$ with $\Hom(U_0, W_0)$. Under this identification $[v_1 \wedge \dots \wedge v_{p-1} \wedge (v_p+t w)]$ corresponds to the homomorphism $tS$ where $S$ is the linear map
\begin{align*}
S\left(\sum_{i=1}^p \alpha_i v_i\right) = \alpha_1 w.
\end{align*}
Then $\ell \cap\Mb = \{ tS : t \in \Rb\}$. Then the first part of the lemma follows from the change of coordinates formula in Subsection~\ref{subsec:affine_charts}.

Next suppose that $X,S \in \Mb$ and $\rank(S)=1$. There exists a basis $v_1, \dots, v_p \in \Rb^{p}$ so that $v_1, \dots, v_{p-1} \in \ker S$ and $Sv_p \neq 0$. Then $X+tS$ corresponds to the subspace
\begin{align*}
\Spanset\{ v_1+X(v_1), \dots, v_{p-1}+X(v_{p-1}), v_p+X(v_p)+tS(v_p) \}
\end{align*}
and hence in the projective model the line
 \begin{align*}
\left[ \Big(v_1+X(v_1)\Big) \wedge \dots \wedge \Big(v_{p-1}+X(v_{p-1})\Big) \wedge \Big(v_p+X(v_p)+tS(v_p)\Big) \right].
\end{align*}
So the closure of $\{ X+tS : t \in \Rb\}$ in $\Pb(\wedge^p \Rb^{p+q})$ is a projective line. 
\end{proof}

\subsection{A Trivial Example}\label{subsec:trivial_ex}

In this subsection we observe that an entire affine chart is an example of a convex divisible domain. Using the matrix model of $\Gr_p(\Rb^{p+q})$ let
\begin{align*}
\Omega = \left\{ \begin{bmatrix} \Id_p \\ X \end{bmatrix} : X \in M_{q,p}(\Rb) \right\} .
\end{align*}
Then 
\begin{align*}
\Aut(\Omega) = \left\{ \begin{bmatrix} \Id_p & 0 \\ Y & \Id_q \end{bmatrix} : Y \in M_{q,p}(\Rb) \right\}
\end{align*}
and 
\begin{align*}
\Gamma = \left\{ \begin{bmatrix} \Id_p & 0 \\ Y & \Id_q \end{bmatrix} : Y \in M_{q,p}(\Zb) \right\}
\end{align*}
is a discrete group which acts freely, properly discontinuously, and cocompactly on $\Omega$. Notice that the quotient $\Gamma \backslash \Omega$ can be identified with the torus of dimension $pq$. 

\part{An invariant metric}

\section{The metric}
\label{sec:metric}

The purpose of this section is to extend Kobayashi's definition of the Hilbert metric to domains in $\Gr_p(\Rb^{p+q})$. 

Suppose that $\Omega \subset \Gr_p(\Rb^{p+q})$ is open and connected. Recall from Subsection~\ref{subsec:hilbert} that $I \subset \Pb(\Rb^2)$ is the open unit interval and $H_I$ is the Hilbert metric on $I$. Using the projective model of the Grassmannians view $\Omega$ as a subset of $\Pb(\wedge^p \Rb^{p+q})$ and let 
\begin{align*}
\Proj(I, \Omega) \subset \Pb( \End(\Rb^2, \wedge^p \Rb^{p+q}))
\end{align*}
be the set of projective maps so that $I \cap \ker T = \emptyset$ and $T(I) \subset \Omega$. Then define a function $\rho_\Omega : \Omega \times \Omega \rightarrow \Rb \cup \{ \infty\}$ as follows: 
\begin{align*}
\rho_\Omega(x,y) := \inf\left\{ H_I(s,t) : \text{ there exists } f \in \Proj(I, \Omega) \text{ with } f(s) = x \text{ and } f(t) = y \right\}.
\end{align*}
We then define
\begin{align*}
K_{\Omega}^{(n)}(x,y)  := \inf\left\{ \sum_{i=0}^{n-1} \rho_{\Omega}(x_i,x_{i+1}) : x= x_0, x_1, \dots, x_{n-1}, x_n=y \in \Omega \right\}.
\end{align*}
In particular $K_{\Omega}^{(n)}(x,y)$ is finite precisely when there is a path in $\Omega$ from $x$ to $y$ consisting of at most $n$ segments of projective lines. Finally we set 
\begin{align*}
K_{\Omega}(x,y) := \lim_{n \rightarrow \infty} K_{\Omega}^{(n)}(x,y).
\end{align*}

\begin{remark} For $x,y \in \Omega$ it is possible to explicitly compute $\rho_\Omega(x,y)$:
\begin{enumerate}
\item if $\dim(x \cap y) < p-1$ then $\rho_\Omega(x,y) = \infty$, 
\item if $\dim(x \cap y) \geq p-1$, let $\ell$ be the projective line in $\Gr_p(\Rb^{p+q})$ containing $x$ and $y$. If $x,y$ are in different connected components of $\ell \cap \Omega$ then $\rho_\Omega(x,y) = \infty$. Finally if $x,y$ are contained in the same connected component $\Oc$ of $\ell \cap \Omega$ then 
\begin{align*}
\rho_\Omega(x,y) = \abs{\log[a;x;y;b]}
\end{align*}
where $a,b$ are the endpoints of $\Oc$. 
\end{enumerate}
\end{remark}

\begin{proposition}\label{prop:basic_d} If $\Omega \subset \Gr_p(\Rb^{p+q})$ is an open connected set then:
\ 
\begin{enumerate}
\item if $\varphi \in \PGL_{p+q}(\Rb)$ then $K_{\Omega}(x,y) = K_{\varphi \Omega}(\varphi x, \varphi y)$ for all $x,y \in \Omega$, 
\item $K_{\Omega}(x,y) \leq K_{\Omega}(x,z) + K_{\Omega}(z,y)$ for any $x,y,z \in \Omega$, 
\item if $\Omega_1 \subset \Omega_2$ then $K_{\Omega_2}(x,y) \leq K_{\Omega_1}(x,y)$ for all $x,y \in \Omega_1$, 
\item for any compact set $K \subset \Omega$ there exists $N > 0$ such that $K_{\Omega}^{(N)}(x,y) < \infty$ for every $x, y \in K$,
\item $K_{\Omega}$ is continuous. 
\end{enumerate}
\end{proposition}

\begin{proof}
Parts (1)-(3) follow from the definition of $K_{\Omega}$ and the invariance of the cross-ratio. 

To establish part (4) it is enough to show the following: for any $x \in \Omega$ there exists an open neighborhood $U$ of $x$ and a number $n =n(p)$ such that $K_{\Omega}^{(n)}(z,y) < \infty$ for any $z,y \in U$. Suppose that $x = [v_1 \wedge v_2 \wedge \dots \wedge v_p]$. Then there exists $\epsilon >0$ such that 
\begin{align*}
U:= \{ [w_1 \wedge w_2 \wedge \dots \wedge w_p] : \norm{v_i - w_i} < \epsilon \text{ for } 1 \leq i \leq p \} \subset \Omega.
\end{align*}
But then clearly $K_{\Omega}^{(p-1)}(z,y) < \infty$ for any $z,y \in U$.

To establish part (5), first observe that 
\begin{align*}
\abs{ K_{\Omega}(x_0,y_0) - K_{\Omega}(x,y)} \leq  K_{\Omega}(x_0, x) + K_{\Omega}(x_0,x)
\end{align*}
so it is enough to show that the map $x \rightarrow K_{\Omega}(x_0,x)$ is continuous at $x_0$. But if $x_0 = [v_1 \wedge v_2 \wedge \dots \wedge v_p]$ then there exists $\epsilon >0$ such that 
\begin{align*}
U:= \{ [w_1 \wedge w_2 \wedge \dots \wedge w_p] : \norm{v_i - w_i} < \epsilon \text{ for } 1 \leq i \leq p\} \subset \Omega.
\end{align*}
But then for $[w_1 \wedge \dots \wedge w_p] \in U$ we have
\begin{align*}
K_{\Omega}(x_0,  [w_1 \wedge w_2 \wedge \dots \wedge w_p]) \leq K_U( x_0, [w_1 \wedge w_2 \wedge \dots \wedge w_p] ) \leq \sum_{i=2}^{p} \log \frac{  \epsilon + \norm{v_i-w_i}}{\epsilon - \norm{v_i-w_i}}
\end{align*}
and so 
\begin{align*}
\lim_{x \rightarrow x_0} K_{\Omega}(x_0, x_n) =0.
\end{align*}
\end{proof}

The above Proposition shows that $K_{\Omega}$ is an $\Aut(\Omega)$-invariant pseudo-metric. We will next show that $K_\Omega$ is a complete metric for certain convex subsets. 

\begin{definition}\label{defn:r_proper} \ \begin{enumerate}
\item Let $\Lc$ be the space of projective lines in $\Pb(\wedge^p \Rb^{p+q})$ which are contained in $\Gr_p(\Rb^{p+q})$, 
\item An open connected set $\Omega \subset \Gr_p(\Rb^{p+q})$ is called \emph{$\Rc$-proper} if 
\begin{align*}
\abs{\ell \setminus \ell \cap \Omega} > 1
\end{align*}
for all $\ell \in \Lc$. 
\end{enumerate}
\end{definition}
 
\begin{example} Suppose $\Mb \subset \Gr_p(\Rb^{p+q})$  is an affine chart and $\Omega$ is a bounded subset of $\Mb$ then $\Omega$ is an $\Rc$-proper subset of $\Gr_p(\Rb^{p+q})$ (see Lemma~\ref{lem:rank_one_lines} above).
\end{example}

\begin{theorem}\label{thm:completeness} 
Suppose $\Mb \subset \Gr_p(\Rb^{p+q})$ is an affine chart and $\Omega \subset \Mb$ is an open convex set. Then the following are equivalent: 
\begin{enumerate}
\item $\Omega$ is $\Rc$-proper 
\item $K_{\Omega}$ is a complete length metric on $\Omega$,
\item $K_{\Omega}$ is a metric on $\Omega$.
\end{enumerate}
\end{theorem}

\begin{remark} The above theorem should be compared to two well known results in real projective geometry and several complex variables:
\begin{enumerate}
\item For a convex set $\Omega \subset \Rb^{d+1}$ the Hilbert metric is complete if and only if $\Omega$ does not contain any real affine lines.
\item For a convex set $\Omega \subset \Cb^{d+1}$ the Kobayashi metric is complete if and only if $\Omega$ does not contain any complex affine lines (Barth ~\cite{B1980}).
\end{enumerate}
\end{remark}

\begin{proof}
Clearly (2) implies (3). Moreover, if there exists a projective line $\ell \in \Lc$ so that 
\begin{align*}
\abs{\ell \setminus \ell \cap \Omega} \leq 1
\end{align*}
then $\rho_\Omega(x,y) = 0$ for all $x,y \in \ell \cap \Omega$. Thus if $\Omega$ is not $\Rc$-proper then $K_\Omega$ is not a metric. Thus (3) implies (1). The proof that (1) implies (2) can be found in Appendix~\ref{sec:proof_complete}.
\end{proof}

The existence of an invariant metric implies that the action of $\Aut(\Omega)$ on $\Omega$ is proper:

\begin{proposition}\label{prop:proper}
Suppose $\Mb \subset \Gr_p(\Rb^{p+q})$ is an affine chart and $\Omega \subset \Mb$ is an open convex set. If $\Omega$ is $\Rc$-proper then 
\begin{enumerate}
\item $\Aut(\Omega)$ is a closed subgroup of $\PGL_{p+q}(\Rb)$,
\item $\Aut(\Omega)$ is a closed subgroup of $\Isom(\Omega, K_{\Omega})$, and
\item  $\Aut(\Omega)$ acts properly on $\Omega$. 
\end{enumerate}
\end{proposition}

\begin{proof}
We first observe that $\Aut(\Omega)$ is closed in $ \PGL_{p+q}(\Rb)$. Suppose that $\varphi_n \in \Aut(\Omega)$ and $\varphi_n \rightarrow \varphi$ in $\PGL_{p+q}(\Rb)$. Then $\varphi(\Omega) \subset \overline{\Omega}$. Since $\Omega$ is convex in an affine chart $\operatorname{int}(\overline{\Omega}) = \Omega$. Then since $\varphi$ induces a homeomorphisms $\Gr_p(\Rb^{p+q}) \rightarrow \Gr_p(\Rb^{p+q})$ we must have 
\begin{align*}
\varphi(\Omega) \subset \operatorname{int}(\overline{\Omega}) = \Omega.
\end{align*}
But the same argument implies that $\varphi^{-1}(\Omega) \subset \Omega$. So $\varphi(\Omega) = \Omega$ and $\varphi \in \Aut(\Omega)$. 

We next show that the action of Aut$(\Omega)$ on $\Omega$ is proper. Suppose that $\varphi_n \in \Aut(\Omega)$ is a sequence of automorphisms so that 
\begin{align*}
\varphi_n x_0 \in \{ y \in \Omega : K_{\Omega}(x_0, y) \leq R \}
\end{align*}
for some $x_0 \in \Omega$ and $R \geq 0$. We need to show that a subsequence of $\varphi_n$ converges in $\PGL_{p+q}(\Rb)$. 

Since $\Aut(\Omega)$ acts by isometries on the metric space $(\Omega, K_{\Omega})$, by the Arzel{\`a}-Ascoli theorem there exists an isometry $f:(\Omega, K_{\Omega}) \rightarrow (\Omega, K_{\Omega})$ and a subsequence $n_k \rightarrow \infty$ so that 
\begin{align*}
f(x) = \lim_{k \rightarrow \infty} \varphi_{n_k} (x)
\end{align*}
for all $x \in \Omega$. Since $f$ is an isometry it is injective. 

Now let $T_k \in \GL(\wedge^p \Rb^{p+q})$ be representatives of $\wedge^p \varphi_{n_k} \in \PGL(\wedge^p \Rb^{p+q})$ with $\norm{T_k}=1$. By passing to another subsequence we can suppose that $T_k \rightarrow T \in \End(\wedge^p \Rb^{p+q})$. Now for $x \in \Omega \setminus \ker T$ we have 
\begin{align*}
T(x) = \lim_{k\rightarrow \infty} \varphi_{n_k}(x) = f(x)
\end{align*}
and so $T$ is injective on $\Omega \setminus \ker T$. But this implies that $T \in \GL(\wedge^p \Rb^{p+q})$. And hence $\varphi_{n_k} \rightarrow \varphi$ in $\PGL_{p+q}(\Rb)$ for some $\varphi$ with $\wedge^p \varphi = [T]$. So $\Aut(\Omega)$ acts properly. 

Notice that the above argument also implies that $\Aut(\Omega)$ is a closed subgroup of $\Isom(\Omega, K_{\Omega})$. 
\end{proof}

\section{Limits in the local Hausdorff topology and rescaling}

Given a set $A \subset \Rb^d$, let $\Nc_{\epsilon}(A)$ denote the \emph{$\epsilon$-neighborhood of $A$} with respect to the Euclidean distance. The \emph{Hausdorff distance} between two bounded sets $A,B$ is given by
\begin{align*}
d_{H}(A,B) = \inf \left\{ \epsilon >0 : A \subset \Nc_{\epsilon}(B) \text{ and } B \subset \Nc_{\epsilon}(A) \right\}.
\end{align*}
Equivalently, 
\begin{align*}
\displaystyle d_H(A,B) = \max\left\{\sup_{a \in A} \inf_{b \in B} \norm{a-b}, \sup_{b \in B} \inf_{a \in A} \norm{a-b} \right\}.
\end{align*}
The Hausdorff distance is a complete metric on the space of compact sets in $\Rb^d$.

The space of closed sets in $\Rb^d$ can be given a topology from the local Hausdorff seminorms. For $R >0$ and a set $A \subset \Rb^d$ let $A^{(R)} := A \cap B_R(0)$. Then define the \emph{local Hausdorff seminorms} by
\begin{align*}
d_H^{(R)}(A,B) := d_H(A^{(R)}, B^{(R)}).
\end{align*}
Finally we say that a sequence of open convex sets $A_n$ converges in the local Hausdorff topology to an open convex set $A$ if $d_H^{(R)}(\overline{A}_n,\overline{A}) \rightarrow 0$ for all $R>0$. 

\begin{theorem}\label{thm:haus_conv}
Let $\Mb$ be an affine chart of $\Gr_p(\Rb^{p+q})$ and suppose $\Omega_n \subset \Mb$ is a sequence of $\Rc$-proper convex open sets converging to a $\Rc$-proper convex open set $\Omega \subset \Mb$ in the local Hausdorff topology. Then 
\begin{align*}
K_{\Omega}(x,y) = \lim_{n \rightarrow \infty} K_{\Omega_n}(x,y)
\end{align*}
for all $x,y \in \Omega$ uniformly on compact sets of $\Omega \times \Omega$.
\end{theorem}

We provide the proof of Theorem~\ref{thm:haus_conv} in Appendix~\ref{sec:proof_of_haus}. 

\begin{theorem}\label{thm:rescaling}
Let $\Mb$ be an affine chart of $\Gr_p(\Rb^{p+q})$ and suppose $\Omega \subset \Mb$ is an $\Rc$-proper open convex subset. Assume in addition that there exists a subgroup $H \leq \Aut(\Omega)$ and a compact set $K \subset \Omega$ such that $H \cdot K = \Omega$. 

If there exists a sequence $A_n \in \Aff(\Mb) \cap \PGL_{p+q}(\Rb)$ such that that $A_n \Omega$ converges in the local Hausdorff topology to an $\Rc$-proper open convex set $\wh{\Omega} \subset \Mb$ then there exists $n_k \rightarrow \infty$ and $h_k \in H$ so that 
\begin{align*}
\phi = \lim_{k \rightarrow \infty} A_{n_k} h_k 
\end{align*}
exists in $\PGL_{p+q}(\Rb)$ and $\wh{\Omega} = \phi(\Omega)$. 
\end{theorem}

\begin{proof}
Fix $y_0 \in \wh{\Omega}$. Then we have $y_0 \in A_n \Omega$ for $n$ sufficiently large. Pick $h_n \in H$ and $k_n \in K$ so that $y_0 = A_n \varphi_n k_n$. Let $T_n: = A_n \varphi_n \in \PGL_{p+q}(\Rb)$. Then 
\begin{align*}
\Omega_n := T_n(\Omega) = A_n(\Omega)
\end{align*} \
is an $\Rc$-proper open convex subset and $T_n$ is an isometry $(\Omega, K_{\Omega}) \rightarrow (\Omega_n, K_{\Omega_n})$. By Theorem~\ref{thm:haus_conv} 
\begin{align*}
K_{\Omega_n} \rightarrow K_{\wh{\Omega}}
\end{align*} 
uniformly on compact sets on $\wh{\Omega}$, so we can pass to a subsequence so that $T_n$ converges uniformly on compact sets to an isometry $T:(\Omega, K_{\Omega}) \rightarrow (\wh{\Omega}, K_{\wh{\Omega}})$.  Since $T$ is an isometry it is injective. On the other hand since the metrics converge and closed metric balls are compact we also see that $T$ is onto. 

Now we can pick a representative $\Phi_n \in \GL(\wedge^p \Rb^{p+q})$ of $\wedge^p T_n \in \PGL(\wedge^p \Rb^{p+q})$ such that $\norm{\Phi_n} =1$. By passing to a subsequence we can assume that $\Phi_n \rightarrow \Phi$ in $\End(\wedge^p \Rb^{p+q})$. The set $\wedge^p \End(\Rb^{p+q}) \subset \End(\wedge^p \Rb^{p+q})$ is closed and so $\Phi = \wedge^p \phi$ for some $\phi \in \End(\Rb^{p+q})$. Moreover $\Phi(x) = T(x)$ for any $x \notin \ker \Phi$. Since $\Gr_p(\Rb^{p+q}) \setminus \ker \Phi$ is an open dense set and $\Omega$ is open, this implies that $\Phi$ is injective on $\Gr_p(\Rb^{p+q}) \setminus \ker \Phi$. It follows that $\Phi \in \GL(\wedge^p \Rb^{p+q})$ and hence $\phi \in \GL_{p+q}(\Rb)$. Finally, we have that $\phi = T$ on $\Omega$, so that $\wh{\Omega} = \phi(\Omega)$. 
\end{proof}

\section{The geometry near the boundary}

For the classical Hilbert metric on a convex divisible domain in real projective space, there are many connections between the shape of the boundary and the behavior of the metric (see for instance~\cite{B2004, B2003a, KN2002}). In a similar spirit, we will prove some basic results connecting the geometry of $K_\Omega$ with the geometry of $\partial \Omega$.

As before, let $\Lc$ be the set of projective lines $\ell \subset \Pb(\wedge^{p} \Rb^{p+q})$ which are contained in $\Gr_{p}(\Rb^{p+q})$. 

\begin{definition}
Suppose $\Omega \subset \Gr_p(\Rb^{p+q})$ is an open connected set. 
\begin{enumerate}
\item  Two points $x,y \in \partial \Omega$ are \emph{adjacent}, denoted $x \sim y$, if either $x=y$ or there exists a projective line $\ell \in \Lc$ so that $x,y$ are contained in a connected component of the interior of $\ell \cap \partial\Omega$ in $\ell$. 
\item The \emph{$\Rc$-face of $x \in \partial \Omega$}, denoted $\Rc F(x)$, is the set of points $y \in \partial \Omega$ where there exists a sequence $x=y_0, y_1, \dots, y_k = y$ with $y_i \sim y_{i+1}$. 
\item A point $x \in \partial \Omega$ is called an \emph{$\Rc$-extreme point} if $\Rc F(x) = \{ x\} $. 
\item Let $\Ext(\Omega) \subset \partial \Omega$ denote the set of $\Rc$-extreme points of $\Omega$.
\end{enumerate}
\end{definition}

As the next two results show this relation on the boundary is connected with the asymptotic geometry of the intrinsic metric. 

\begin{proposition}
Suppose $\Mb \subset \Gr_p(\Rb^{p+q})$ is an affine chart and $\Omega \subset \Mb$ is an $\Rc$-proper open convex set. If $x_n,y_n \in \Omega$ are sequences so that $x_n \rightarrow x\in \partial \Omega$, $y_n \rightarrow y \in \partial \Omega$, and there exists $N \geq 0$ so that 
\begin{align*}
\liminf_{n \rightarrow \infty} K_{\Omega}^{(N)}(x_n, y_n) < \infty
\end{align*}
then $\Rc F(x) = \Rc F(y)$. 
\end{proposition}

\begin{proof}
By induction and passing to a subsequence, it is enough to consider the case in which 
\begin{align*}
\lim_{n \rightarrow \infty} K_{\Omega}^{(1)}(x_n, y_n) = \lim_{n \rightarrow \infty} \rho_\Omega(x_n, y_n)< \infty
\end{align*}
and $x \neq y$. For each $n$ let $\ell_n$ be the projective line containing $x_n$ and $y_n$. Also let $\{a_n, b_n\} = \ell_n \cap \partial \Omega$ ordered $a_n,x_n, y_n, b_n$ along $\ell_n$. Then 
\begin{align*}
\rho_\Omega(x_n, y_n) = \log \frac{ \abs{x_n-b_n} \abs{y_n-a_n}}{\abs{x_n-a_n}\abs{y_n-b_n}}.
\end{align*}
By passing to a subsequence we can suppose that $a_n \rightarrow a$ and $b_n \rightarrow b$. Then by the hypothesis we must have that $a \neq x$ and $b \neq y$. So $x \sim y$. 
\end{proof}

\begin{corollary}\label{cor:asym_faces}
Suppose $\Mb \subset \Gr_p(\Rb^{p+q})$ is an affine chart, $\Omega \subset \Mb$ is an $\Rc$-proper open convex set, and $\Aut(\Omega)$ acts cocompactly on $\Omega$. If $x_n,y_n \in \Omega$ are sequences so that $x_n \rightarrow x \in \partial \Omega$, $y_n \rightarrow y \in \partial \Omega$, and 
\begin{align*}
\liminf_{n \rightarrow \infty} K_{\Omega}(x_n, y_n) < \infty
\end{align*}
then $\Rc F(x) = \Rc F(y)$. 
\end{corollary}

\begin{proof}
By passing to a subsequence we can suppose that 
\begin{align*}
M = \sup_{n \in \Nb} K_{\Omega}(x_n, y_n)  < \infty. 
\end{align*}
For $R\geq 0$ and $x\in \Omega$, let $B_R(x)$ denote the ball of radius $R$ and center $x$ with respect to the metric $K_\Omega$. Since $\Aut(\Omega)$ acts cocompactly on $\Omega$ there exists $R \geq 0$ so that 
\begin{align*}
\Aut(\Omega) \cdot B_R(x_0) = \Omega.
\end{align*}
Let $B:=B_{R+M}(x_0)$ be the ball with center $x_0$ and radius $R+M$. By compactness of $B$ and Proposition~\ref{prop:basic_d}, we know there exists $N >0$ so that 
\begin{align*}
\sup_{x,y \in B} K_{\Omega}^{(N)}(x,y) < \infty
\end{align*} 
for all $x,y \in B$. But this implies that 
\begin{align*}
 \sup_{n \in \Nb} K_{\Omega}^{(N)}(x_n, y_n) < \infty
\end{align*}
because for any $n \in \Nb$ there exists some $\varphi \in \Aut(\Omega)$ so that $\varphi x_n, \varphi y_n \in B$. 
\end{proof}

\part{The automorphism group is non-discrete}

\section{Extreme points and symmetry} 

\subsection{The geometry of extreme points}  

In this subsection we provide a number of characterizations of $\Rc$-extreme points for domains $\Omega \subset \Gr_p(\Rb^{2p})$ where $\Aut(\Omega)$ acts cocompactly. But first a few definitions. 

Suppose $\Omega$ is a convex set in a vector space and $x \in \partial \Omega$, then the \emph{tangent cone of $\Omega$ at $x$} is the set 
\begin{align*}
\mathcal{T}\mathcal{C}_x \Omega :=x+ \bigcup_{t > 0} t(\Omega - x).
\end{align*}
Notice that the sets $x+t(\Omega - x)$ converge to $\mathcal{T}\mathcal{C}_x \Omega$ in the local Hausdorff topology as $t \rightarrow \infty$. 

We will also define natural hypersurfaces in $\Gr_p(\Rb^{p+q})$. 

\begin{definition}
Given $\xi \in \Gr_q(\Rb^{p+q})$ define the hypersurface 
\begin{align*}
Z_{\xi} := \{ x \in \Gr_p(\Rb^{p+q}) : x \cap \xi \neq (0) \}.
\end{align*}
\end{definition}

\begin{remark}
In the case in which $p=1$, then $Z_\xi \subset \Pb(\Rb^{1+q}) = \Gr_1(\Rb^{1+q})$ is the image of $\xi$ in $\Pb(\Rb^{q+1})$. In particular, if a set $\Omega \subset \Pb(\Rb^d)$ is convex and bounded in an affine chart then for any $x \in \partial \Omega$ there exists $\xi \in \Gr_{d-1}(\Rb^d)$ so that $x \in Z_\xi$ and $Z_{\xi} \cap \Omega = \emptyset$. 
\end{remark}

These hypersurfaces were used in~\cite{Z2015} to show that symmetry implies a type of convexity:

\begin{theorem}\cite[Theorem 1.7]{Z2015}
If $\Omega \subset \Gr_p(\Rb^{p+q})$ is a bounded connected open subset of some affine chart and $\Aut(\Omega)$ acts cocompactly on $\Omega$ then for all $x \in \partial \Omega$ there exists $\xi \in \Gr_q(\Rb^{p+q})$ so that $x \in Z_\xi$ and $Z_{\xi} \cap \Omega = \emptyset$. 
\end{theorem}

With these notations we will prove the following:

\begin{theorem}\label{thm:extreme} Suppose $p > 1$, $\Mb \subset \Gr_p(\Rb^{2p})$ is an affine chart, $\Omega$ is a bounded open convex subset of $\Mb$, and $\Aut(\Omega)$ acts cocompactly on $\Omega$. If $e \in \partial \Omega$ then the following are equivalent:
\begin{enumerate}
\item $e \in \partial \Omega$ is an $\Rc$-extreme point, 
\item $Z_e \cap \Omega = \emptyset$, 
\item $\Tc \Cc_e \Omega$ is an $\Rc$-proper cone, 
\item there exist $\varphi_n \in \Aut(\Omega)$ and representatives $\wh{\varphi}_n \in \GL(\wedge^p \Rb^{2p})$ so that $\wh{\varphi}_n \rightarrow S$ in $\End(\wedge^p \Rb^{2p})$ and $\operatorname{Im}(S) = e$.
\end{enumerate}
\end{theorem}

\begin{remark} Part (4) fails for convex divisible domains in real projective space. In particular by a result of Benoist~\cite{B2006}: if $\Omega \subset \Pb(\Rb^4)$ is a convex divisible domain and $x \in \partial \Omega$ then there exists $\varphi_n \in \Aut(\Omega)$ and representatives $\wh{\varphi}_n \in \GL_4(\Rb)$ so that $\wh{\varphi}_n \rightarrow S$ in $\End(\Rb^{4})$ and $\operatorname{Im}(S) = x$. However, there are examples of convex divisible domains in $\Pb(\Rb^4)$ whose boundary contains non-extreme points (see~\cite{B2006} and~\cite{BDL2015}).\end{remark}

\begin{proof}
We first show that $(1) \Rightarrow (4)$. Suppose that $e \in \partial \Omega$ is an $\Rc$-extreme point. Pick a sequence $x_n \in \Omega$ so that $x_n \rightarrow e$. Since $\Aut(\Omega)$ acts cocompactly on $\Omega$ we can find $R \geq 0$ and $\varphi_n \in \Aut(\Omega)$ so that 
\begin{align*}
K_{\Omega}(x_n, \varphi_n x_0) \leq R
\end{align*}
for all $n \geq 0$. Now for any $x \in \Omega$ we have 
\begin{align*}
K_{\Omega}(\varphi_n x,x_n) \leq K_{\Omega}(\varphi_n x,\varphi_n x_0) + R = K_{\Omega}(x,x_0)+R
\end{align*}
and so by Corollary~\ref{cor:asym_faces} we see that $\varphi_n x \rightarrow e$. Next pick representatives $\wh{\varphi}_n \in \GL(\wedge^p \Rb^{2p})$ of $\wedge^p \varphi_n$ so that $\norm{\wh{\varphi}_n}=1$. By passing to a subsequence we can suppose that $\wh{\varphi}_n \rightarrow S$ in $\End(\wedge^p \Rb^{2p})$. Now if $x \in \Oc:=\Gr_{p}(\Rb^{2p}) \setminus \ker S$ then $S(x) = \lim_{n \rightarrow \infty} \varphi_n x$. Since $\Oc$ is open and dense, we see that $\Omega \cap \Oc$ is dense in $\Omega$. In particular $\Omega \cap \Oc$ contains a basis of $\wedge^p \Rb^{2p}$. However for every $x \in \Omega \cap \Oc$ we have $S(x) = e$. So $\operatorname{Im}(S) = e$. So $(1) \Rightarrow (4)$. 

We next show that $(4) \Rightarrow (2)$. So suppose there exists $\varphi_n \in \Aut(\Omega)$ and representatives $\wh{\varphi}_n \in \GL(\wedge^p \Rb^{2p})$ so that $\wh{\varphi}_n \rightarrow S$ in $\End(\wedge^p \Rb^{2p})$ and $\operatorname{Im}(S) = e$. Notice that if $x \in \Oc:=\Gr_p(\Rb^{2p}) \setminus \ker S$ then $S(x) = \lim_{n \rightarrow \infty} \varphi_n(x)$. Now define the set 
\begin{align*}
\Omega^* := \{ \xi \in \Gr_p(\Rb^{2p}) : Z_{\xi} \cap \Omega = \emptyset\}.
\end{align*}
Since $\Omega$ is open, $\Omega^*$ is compact. Moreover since $\Omega$ is bounded in an affine chart $\Omega^*$ has non-empty interior: $\Mb = \Gr_p(\Rb^{2p}) \setminus Z_{\xi}$ for some $\xi$ and since $\Omega$ is bounded in $\Mb$ we see that $\Omega^*$ contains an open neighborhood of $\xi$. In particular, $\Omega^* \cap \Oc$ is non-empty. But then for $\eta \in \Omega^* \cap \Oc$ we have $e=S(\eta) = \lim_{n \rightarrow \infty} \varphi_n(\eta)$. Since $\Omega^*$ is $\Aut(\Omega)$-invariant we then see that $e \in \Omega^*$.  So $(4) \Rightarrow (2)$.

We next show that $(2) \Rightarrow (3)$. So suppose that $e \in \partial \Omega$ and $Z_e \cap \Omega = \emptyset$. We can assume that 
\begin{align*}
\Omega \subset \Mb:= \left\{ \begin{bmatrix} \Id_p \\ X \end{bmatrix} : X \in M_{p,p}(\Rb) \right\}
\end{align*}
and $e=0$ in $\Mb$. Then since $Z_e \cap \Omega =  \emptyset$ we see that 
\begin{align*}
\Omega \subset \left\{ \begin{bmatrix} \Id_p \\ X \end{bmatrix} : \det(X) \neq 0\right\}.
\end{align*}
Since $\Omega$ is connected, by making an affine transformation, we may assume that 
\begin{align*}
\Omega \subset \left\{ \begin{bmatrix} \Id_p \\ X \end{bmatrix} :  \det(X) >0  \right\}.
\end{align*}
Then, since $\mathcal{T}\mathcal{C}_0 \Omega$ is open, we see that 
\begin{align*}
\mathcal{T}\mathcal{C}_0 \Omega \subset \left\{ \begin{bmatrix} \Id_p \\ X \end{bmatrix} :  \det(X) >0  \right\}.
\end{align*}
Now suppose for a contradiction that $\mathcal{T}\mathcal{C}_0 \Omega$ is not $\Rc$-proper. Then by Lemma~\ref{lem:rank_one_lines} and convexity there exists a rank one endomorphism $S$ so that 
\begin{align*}
\{ \begin{bmatrix} \Id_p & T + tS\end{bmatrix}^t : t \in \Rb\} \subset \mathcal{T}\mathcal{C}_0 \Omega 
\end{align*}
whenever $ \begin{bmatrix} \Id_p & T\end{bmatrix}^t  \in \mathcal{T}\mathcal{C}_0 \Omega$. So
\begin{align*}
\det(T+tS) >0
\end{align*}
for any $ \begin{bmatrix} \Id_p & T \end{bmatrix}^t  \in \mathcal{T}\mathcal{C}_0 \Omega$ and $t \in \Rb$. Now 
\begin{align*}
\det(T+tS) = \det(T) \det(\Id_p + t T^{-1} S) = \det(T)(1+t\tr(T^{-1} S))
\end{align*}
since $T^{-1}S$ has rank one. But since $\mathcal{T}\mathcal{C}_0 \Omega$ is open there exists some $ \begin{bmatrix} \Id_p & T_0\end{bmatrix}^t  \in \mathcal{T}\mathcal{C}_0 \Omega$ so that $\tr T_0^{-1}S$ is non-zero. But then 
\begin{align*}
\det(T_0+tS)=0
\end{align*}
when $t=-(\tr T_0^{-1}S)^{-1}$. So we have a  contradiction and so $(2) \Rightarrow (3)$. 

Finally we show that $(3) \Rightarrow (1)$. So suppose that $e \in \partial \Omega$ and $\mathcal{T}\mathcal{C}_e \Omega$ is $\Rc$-proper. If $e \in \partial \Omega$ is not an $\Rc$-extreme point then $\overline{\Tc \Cc_e \Omega}$ contains an entire rank one line. Since $\Tc \Cc_e \Omega$ is convex and open this implies that $\Tc \Cc_e \Omega$ contains an entire rank one line and so $\Tc \Cc_e \Omega$ is not $\Rc$-proper.  
\end{proof}

\begin{corollary}\label{cor:ext_closed} Suppose $p > 1$, $\Mb \subset \Gr_p(\Rb^{2p})$ is an affine chart, $\Omega$ is a bounded open convex subset of $\Mb$, and $\Aut(\Omega)$ acts cocompactly on $\Omega$. Then $\Ext(\Omega) \subset \partial \Omega$ is closed. \end{corollary}

\begin{remark} This corollary fails for convex divisible domains in real projective space. In particular by a result of Benoist~\cite{B2006}: if $\Omega \subset \Pb(\Rb^4)$ is a convex divisible domain then the extreme points of $\Omega$ are dense in $\partial \Omega$.  However, there are examples of convex divisible domains in $\Pb(\Rb^4)$ whose boundary contains non-extreme points (see~\cite{B2006} and~\cite{BDL2015}).\end{remark}

\begin{proof}
By the above proposition, the set of extreme points coincides with 
\begin{align*}
\{ e \in \partial \Omega : Z_e \cap \Omega = \emptyset\}
\end{align*}
which is obviously closed. 
\end{proof}

\subsection{Constructing extreme points}

\begin{proposition}\label{prop:ext_spanning}
Suppose $\Mb \subset \Gr_p(\Rb^{p+q})$ is an affine chart and $\Omega \subset \Mb$ is an open bounded convex set. Then $\Ext(\Omega)$ spans $\wedge^p \Rb^{p+q}$. 
\end{proposition}

\begin{proof} Identify $\Mb$ with $M_{q,p}(\Rb)$. For $x \in \partial \Omega$ let 
\begin{align*}
V_x = x+ \Spanset \{ v \in M_{q,p}(\Rb):  v +x \text{ is adjacent to } x \} \subset M_{q,p}(\Rb).
\end{align*}
Notice that $x \in \partial \Omega$ is an $\Rc$-extreme point if and only if $\dimension V_x =0$. 

Now, since rank one lines in $M_{q,p}(\Rb)$ are mapped to projective lines in $\Pb(\wedge^p \Rb^{p+q})$, we have the following: if $v$ is a rank one matrix, $t < 0 < s$, and $a,b,c \in\Pb( \wedge^p \Rb^{p+q})$ are the images of $x+tv, x, x+sv \in M_{q,p}(\Rb)$ respectively then 
\begin{align*}
b \subset a+b.
\end{align*} 
Thus is enough to show: for any $x \in \partial\Omega$ with $\dim V_x > 0$ there exists a rank one matrix $v \in M_{q,p}(\Rb)$ and $t < 0 <s$ so that $x+tv, x+sv \in \partial \Omega$ and 
\begin{align*}
\dimension V_{x+tv}, V_{x+sv} < \dimension V_x.
\end{align*}

Let $F_x = \partial \Omega \cap V_x$. This is a convex set which is open in $V_x$. We claim that $V_y \subset V_x$ for $y \in F_x \cap \partial \Omega$. To see this suppose that $v+y$ is adjacent to $y$. Then there exists $\epsilon > 0$ so that $t v + y \in \partial \Omega$ for $t \in (-\epsilon, 1+\epsilon)$. Moreover, since $y \in \partial \Omega \cap V_x$ there exists $\delta > 0$ so that $\lambda x + (1-\lambda)y \in \partial\Omega$ for $[0, 1+\delta]$. Then by convexity, there exists $\epsilon_1>0$ so that $x+tv \in \partial \Omega$ for $t \in (-\epsilon_1, \epsilon_1)$. Thus $V_y \subset V_x$. This implies that if $y \in \partial F_x$ (viewing $F_x$ as an open set in $V_x$) then $\dimension V_y < \dimension V_x$. 

So for $x \in \partial \Omega$ and $\dim V_x > 0$, pick a rank one matrix $v$ so that $x + \Rb v \subset V_x$. Then if
\begin{align*}
\{x+sv, x+tv \} = \partial F_x \cap (x + \Rb v)
\end{align*}
we have 
\begin{align*}
\dimension V_{x+tv}, V_{x+sv} < \dimension V_x
\end{align*}
and the proof is complete by the remarks above.
\end{proof}

Let $V$ be real vector space of dimension $d<\infty$, $\varphi \in \PGL(V)$, and $\overline{\varphi} \in \GL(V)$ be a representative of $\varphi$ with $\det(\overline{\varphi})=\pm 1$. Next let
\begin{align*}
\sigma_1(\varphi) \leq \sigma_2(\varphi) \leq \dots \leq \sigma_{d}(\varphi)
\end{align*}
be the absolute values of the eigenvalues (counted with multiplicity) of $\overline{\varphi}$ (notice that this does to depend on the choice of $\overline{\varphi}$). Let $m^+(\varphi)$ be the size of the largest Jordan block of $\overline{\varphi}$ whose corresponding eigenvalue has absolute value $\sigma_{d+1}(\varphi)$. Next let $E^+(\varphi)$ be the span of the eigenvectors of $\overline{\varphi}$ whose eigenvalue have absolute value $\sigma_{d+1}(\varphi)$ and are part of a Jordan block with size $m^+(\varphi)$. Also define $E^-(\varphi) = E^+(\varphi^{-1})$.

Given $y \in \Pb(V)$ let $L(\varphi,y) \subset \Pb(V)$ denote the limit points of the sequence $\{\varphi^n y\}_{n \in \Nb}$. With this notation we have the following observation:

\begin{proposition}
\label{prop:attracting}
Suppose $\varphi \in \PGL(V)$ and $\{ \varphi^n\}_{ n \in \Nb} \subset \PGL(V)$ is unbounded, then there exists a proper projective subspace $H \subsetneq \Pb(V)$ such that $L(\varphi,y) \subset [E^+(\varphi)]$ for all $y \in \Pb(V) \setminus H$.
\end{proposition}

\begin{proof} This follows easily once $\overline{\varphi}$ is written in Jordan normal form.
\end{proof}

\begin{corollary}\label{cor:ext_attracting}
Suppose $\Omega$ is an open connected set of $\Gr_p(\Rb^{2p})$, there exists an affine chart which contains $\Omega$ as an bounded convex set, and $\Aut(\Omega)$ acts cocompactly on $\Omega$. If $\varphi \in \Aut(\Omega)$ and $ \varphi^n\rightarrow \infty$ in $\PGL_{2p}(\Rb)$ then $E^+(\wedge^p \varphi) \cap \partial \Omega$ is non-empty and contains an $\Rc$-extreme point. 
\end{corollary}

\begin{proof} Let $H \subset \Pb(\wedge^p \Rb^{2p})$ be as in the above proposition for $\wedge^p \varphi$. Since the set of $\Rc$-extreme points of $\partial \Omega$ span $\Pb(\wedge^p \Rb^{2p})$ there exists some $\Rc$-extreme point $e \in \partial \Omega$ so that $e \notin H$. Then any limit point of $\varphi^n e$ is in $E^+(\wedge^p \varphi)$ and is also an $\Rc$-extreme point by Corollary~\ref{cor:ext_closed}.

 \end{proof}
 
 \subsection{Finding symmetry}

\begin{corollary}\label{cor:rescaling}
Suppose $\Omega \subset \Gr_p(\Rb^{2p})$ is an $\Rc$-proper open convex set in the affine chart 
\begin{align*}
\Mb = \left\{ \begin{bmatrix} \Id_p \\ X \end{bmatrix} : X \in M_{p,p}(\Rb) \right\}
\end{align*}
and $H \leq \Aut(\Omega)$ acts cocompactly on $\Omega$. If $e=\begin{bmatrix} \Id_p \\ X_0\end{bmatrix} \in \partial \Omega$ is an $\Rc$-extreme point then there exists $h_n \in H$ and $t_n \rightarrow \infty$ so that 
\begin{align*}
\varphi = \lim_{n \rightarrow \infty} \begin{bmatrix} \Id_p & 0 \\ (1-e^{t_n})X_0 & e^{t_n} \Id_p \end{bmatrix} h_n 
\end{align*}
exists in $\PGL_{2p}(\Rb)$ and $\varphi(\Omega) = \mathcal{T}\mathcal{C}_e \Omega$. In particular, $\Omega$ is invariant under the one-parameter group 
\begin{align*}
\varphi^{-1} \left\{ \begin{bmatrix} \Id_p & 0 \\ (1-e^t)X_0 & e^t \Id_p \end{bmatrix} : t \in \Rb \right\}\varphi.
\end{align*}
\end{corollary}

\begin{proof}
Let 
\begin{align*}
A_t = \begin{bmatrix} \Id_p & 0 \\ (1-e^{t})X_0 & e^{t} \Id_p \end{bmatrix} 
\end{align*}
then 
\begin{align*}
A_t \cdot \begin{bmatrix} \Id_p \\ X \end{bmatrix}  = \begin{bmatrix} \Id_p \\ e^t(X-X_0)+X_0 \end{bmatrix}.
\end{align*}
So $A_t \in \Aff(\Mb) \cap \PGL_{2p}(\Rb)$ and $A_t\Omega$ converges in the local Hausdorff topology to $\Tc \Cc_e \Omega$ as $t \rightarrow \infty$. So the corollary follows  from Theorem~\ref{thm:rescaling} and Theorem~\ref{thm:extreme}.
\end{proof}

\part{The automorphism group is simple}

\section{Initial reduction}
\label{sec:initial}

For the rest of this section suppose $p > 1$, $\Mb \subset \Gr_p(\Rb^{2p})$ is an affine chart, $\Omega \subset \Mb$ is a bounded convex open subset of $\Mb$, and there exists a discrete group $\Gamma \leq \Aut(\Omega)$ so that $\Gamma$ acts cocompactly on $\Omega$.

Set $G:=\Aut(\Omega)$ and let $G^0$ be the connected component of the identity of $G$. By Corollary~\ref{cor:rescaling}, we know that $G^0\neq 1$. The goal of this section is to use the fact that $G^0\neq 1$ to obtain that either $G^0$ is simple and acts transitively on $\Omega$, or we are in one of four very constrained situations (Cases (1)-(4) in Theorem \ref{thm:reduction} below). In Sections \ref{sec:center}, \ref{sec:unipotent}, and \ref{sec:p_2}, we will prove that Cases (1)-(4) cannot occur. 

Since the statement of Theorem \ref{thm:reduction} may seem unmotivated at first, let us sketch the argument. Let $G^{sol} \leq G^0$ be the solvable radical of $G$ (that is, the maximal connected, closed, normal, solvable subgroup of $G$), let $N$ be the nilpotent radical of $G^{sol}$ (that is, the maximal connected normal closed nilpotent subgroup of $G^{sol}$), and $Z$ the center of $N$. 

First suppose that $G^0$ is not semisimple.  Then $Z \neq 1$ and is normalized by $\Gamma$. There are two cases:
	\begin{enumerate}
	\item If $Z$ only consists of unipotent elements, we will show that $\Gamma$ intersects some normal unipotent subgroup in a lattice. This corresponds to Case (2) in Theorem~\ref{thm:reduction}.
		\item Otherwise, we show that a finite index subgroup of $\Gamma$ centralizes some semisimple torus in the Zariski closure of $Z$. This corresponds to Case (1) in Theorem~\ref{thm:reduction}.
	\end{enumerate}
	
Suppose now that $G^0$ is semisimple. We want to show $G^0$ actually has to be simple and acts transitively on $\Omega$. We do this by using the virtual cohomological dimension vcd($\Gamma$) of $\Gamma$ (see below for more information). We know that vcd$(\Gamma)=\dim(\Omega)=p^2$. Then we relate vcd($\Gamma$) to the structure of $G^0$ to show that $G$ has to have finitely many components, and $G^0$ is simple. This latter argument only fails if $p=2$, in which case we obtain very specific information on the structure of $G^0$ and its action on $\Omega$ (Cases (3) and (4) in Theorem~\ref{thm:reduction} below). 

We start with the following lemma.
\begin{lemma}
With the notation above, $\Gamma$ is a cocompact lattice in $G$ and $\Gamma_0 : = \Gamma \cap G^0$ is a cocompact lattice in $G^0$. 
\end{lemma}

\begin{proof} Since $\Gamma$ acts cocompactly on $\Omega$ and $G$ acts properly on $\Omega$ (see Proposition~\ref{prop:proper}) we see that $\Gamma \leq G$ is a cocompact lattice. Since $G^0 \leq G$ is a connected component the set $\Gamma \cdot G^0$ is closed in $G$. So $\Gamma_0 \backslash G^0$ is closed in $\Gamma \backslash G$. Then since $\Gamma \backslash G$ is compact so is $\Gamma_0 \backslash G^0$.
\end{proof}

\begin{theorem}\label{thm:reduction} With the notation above, at least one of the following holds:
\begin{enumerate}
\item a finite index subgroup of $\Gamma$ has non-trivial centralizer in $\PGL_{2p}(\Rb)$, \label{item:centralizer}
\item there exists a nontrivial abelian normal unipotent group $U \leq G$ such that $\Gamma \cap U$ is a cocompact lattice in $U$, 
\item $p=2$ and there exists a finite index subgroup $G'$ of $G$ such that $G'=G^0\times \Lambda$ for some discrete group $\Lambda$. Further up to conjugation
\begin{align*}
G^0 = \left\{ \begin{bmatrix} A & 0 \\ 0 & A \end{bmatrix} : A \in \SL_2(\Rb) \right\}
\end{align*}
and 
\begin{align*}
\Lambda \leq \left\{ \begin{bmatrix} a \Id_2 & b \Id_2 \\ c\Id_2 & d\Id_2 \end{bmatrix} : ad-bc = 1\right\}.
\end{align*}
\item $p=2$, $G^0 \leq G$ has finite index and acts transitively on $\Omega$, and up to conjugation
\begin{align*}
G^0=\left\{ \begin{bmatrix} aA & bA \\ cA & dA \end{bmatrix} : A \in \SL_2(\Rb), ad-bc = 1\right\}.
\end{align*}
\item $G^0$ is a simple Lie group with trivial center that acts transitively on $\Omega$. \label{item:simple}
\end{enumerate}
\label{thm:initial}
\end{theorem}

The rest of this section will be devoted to the proof of Theorem \ref{thm:reduction}. We will assume that case (1), (2), (3), and (4) do not hold and show that case (5) occurs.

\begin{lemma}
With the notation above, $\Gamma^0 \cap Z$ is a cocompact lattice in $Z$.
\end{lemma}

\begin{proof}
Let $G^{ss} \leq G$ be a semisimple subgroup so that $G^0 = G^{ss} G^{sol}$ is a Levi-Malcev decomposition of $G^0$. Then let $\sigma: G^{ss} \rightarrow \Aut(G^{sol})$ be the action of $G^{ss}$ by conjugation on $G^{sol}$. If $\ker \sigma$ has no compact factors in its identity component then $\Gamma^0 \cap N$ is a cocompact lattice in $N$ (see~\cite[Theorem 1.3.(i)]{G2015}). In this case,  $\Gamma^0 \cap Z \leq Z$ is a cocompact lattice by~\cite[Proposition 2.17]{R1972}.

Therefore it suffices to show $\ker \sigma$ contains no compact factors. Since $\ker \sigma \leq G^{ss}$ we see that $\ker \sigma$ is semisimple. So there is a unique maximal compact factor $K_0$ in $\ker \sigma$. Assume for a contradiction that $\dim K_0 > 0$. Then $K_0$ is also a factor of $G^{ss}$ and hence $G^0$ which is impossible by the following argument of Farb and Weinberger~\cite[Claim II]{FW2008}. Let us sketch this proof for completeness. 

Let $K$ be a maximal compact factor of $G^0$. Since $\dim K_0 > 0$ we see that $\dim K >0$. Consider the natural quotient map $\Omega\rightarrow\Omega\slash K$. Since $\Gamma$ permutes the maximal compact factors of $G^0$, we see that a finite index subgroup of $\Gamma$ normalizes $K$. Then it is not hard to see that there is a continuous quasi-isometric inverse $\Omega\slash K\rightarrow \Omega$ to this quotient map. Consider the maps induced by the composition
	$$\Omega\rightarrow \Omega\slash K\rightarrow \Omega$$
on locally finite simplicial homology. On the one hand, since this composition is a bounded distance from the identity map, the induced map on locally finite simplicial homology is the identity map. On the other hand, since $\Omega$ is the universal cover of a closed aspherical manifold,  there is a fundamental class in top degree. But since $\dim K>0$, the image of this fundamental class in $H_\ast(\Omega\slash K)$ vanishes. This is a contradiction. For full details, see the proof of Claim II in~\cite{FW2008}.
\end{proof}

\begin{lemma} $G^0$ is semisimple. \end{lemma}

\begin{proof} 
As above let $N$ be the nilpotent radical of $G^{sol}$ and $Z$ the center of $N$. If $N =1$ then $G^0$ is semisimple. So suppose for a contradiction that $N \neq 1$. Then $Z \neq 1$. Next let $C$ be the Zariski closure of $ Z$ in $\PSL_{2p}(\Rb)$ and let $C^0$ be the connected component of the identity in $C$. Since $G$ normalizes $Z$ it also normalizes $C$ and $C^0$. 

%Using the proof of Proposition 3.3 in~\cite{FW2008} we see that either $G^0$ is semisimple or $\Gamma$ has a infinite normal abelian subgroup $N$. Let $C$ be the Zariski closure of $N$ in $\PSL_{2p}(\Rb)$. By passing to a finite index subgroup of $N$ we may assume that $C$ is connected. Since $N$ is Zariski dense in $C$ and $C$ is abelian, by~\cite[Theorem 2.1]{R1972} $N$ is a cocompact lattice in $C$. Since $N \leq \Gamma \cap C$ we see that $\Gamma \cap C$ is also a cocompact lattice in $C$.

Since $Z$ is abelian so is $C^0$. Then since $C^0$ is an abelian real algebraic group, we can write
\begin{align*}
C^0 =C_{ss} C_{u}
\end{align*}
where $C_{ss}$ is the subset of semisimple elements in $C^0$ and $C_u$ is the subset of unipotent elements of $C^0$ (see for instance~\cite[Theorem 4.7]{B1991}). By~\cite[Corollary 4.4.]{B1991} both $C_{ss}$ and $C_u$ are actually groups. Since $G$ normalizes $C^0$ is also normalizes $C_{ss}$ and $C_u$.

If $C_{ss}=1$ then each element of $C^0$ is unipotent and thus each element of $Z$ is unipotent. Thus we are in case (2), which is a contradiction. Therefore we have $C_{ss} \neq 1$. But the normalizer of any semisimple torus $T$ in $PGL_{2p}(\Rb)$ contains the centralizer of $T$ with finite index~\cite[Corollary 8.10.2]{B1991}, so we know that a finite index subgroup of $G$ centralizes $C_{ss}$. Hence we are in case (1) which contradicts our initial assumption. Thus $G^0$ is semisimple.
\end{proof}

Now let 
\begin{align*}
\SL^{\pm}_{2p}(\Rb) = \{ g \in \GL_{2p}(\Rb) : \det g = \pm 1\}.
\end{align*}
Then let $\wh{G}$ be the inverse image of $G$ under the map $\pi: \SL^{\pm}_{2p}(\Rb)\rightarrow \PGL_{2p}(\Rb)$ and let $\wh{G}^0$ be the connected component of the identity of $\wh{G}$.

\begin{lemma}
$G^0$ has trivial center. 
\end{lemma} 

\begin{proof} Let $Z$ be the center of $G^0$. We first claim that $Z$ is finite. Since $\wh{G}^0 \leq \SL_{2p}(\Rb)$ and $\pi:  \SL_{2p}(\Rb)\rightarrow \PSL_{2p}(\Rb)$ is a double cover we see that $\wh{G}^0$ is locally isomorphic to $G^0$. So $\wh{G}^0$ is a connected semisimple linear group and hence has finite center. Therefore to show that $Z$ is finite, it suffices to show that $\pi^{-1}(Z)$ is contained in the center of $\wh{G}^0$. But now if $z \in \pi^{-1}(Z)$ and $g \in \wh{G}^0$ then 
\begin{align*}
gzg^{-1}z^{-1} \in \{ \Id_{2p}, -\Id_{2p} \}.
\end{align*}
Since $\wh{G}^0$ is connected we then see that 
\begin{align*}
gzg^{-1}z^{-1} = \Id_{2p}
\end{align*}
for all $z \in \pi^{-1}(Z)$ and $g \in \wh{G}^0$. Thus $\pi^{-1}(Z)$ is contained in the center of $\wh{G}^0$, as desired.

Since $G$ normalizes $G^0$, $G$ also normalizes $Z$. Thus, since $Z$ is finite, a finite index subgroup of $G$ centralizes $Z$. Thus if $Z \neq 1$ we are in case (1).
\end{proof}

Next we use an argument of Farb and Weinberger to deduce: 

\begin{lemma}\cite[Proposition 3.1]{FW2008} $G$ has a finite index subgroup $G'$ such that $G'\cong G^0\times\Lambda$ for some discrete group $\Lambda$ and $\Gamma$ has a finite index subgroup $\Gamma'$ such that $\Gamma'\cong \Gamma_0\times \Lambda$. Moreover, by possibly passing to a finite index subgroup of $G^\prime$ we may assume that $\Lambda$ is either trivial or infinite. 
\end{lemma} 

\begin{remark}
The above Lemma follows from the ``triviality of the extension'' part of the proof of Proposition 3.1 in~\cite{FW2008}. This part of their proof only involves the groups and not the Riemannian metric in the statement of Proposition 3.1. In particular, this part of the argument adapts to our situation verbatim. 
\end{remark}

Now decompose $\wh{G}^0\curvearrowright \Rb^{2p}$ as a direct sum of irreducible representations of the semisimple group $\wh{G}^0$: 
	\begin{equation}
	\Rb^{2p}\cong \bigoplus_{\rho} V_{\rho}^{n_\rho}.
	\label{eq:decomp}
	\end{equation}
Here the direct sum is over nonisomorphic irreducible representations $\rho$ of $\wh{G}^0$ and $n_\rho \geq 0$ is the multiplicity of $\rho$. Now since $\wh{G}$ normalizes $\wh{G}^0$ we see that $\wh{G}$ preserves each $V_{\rho}^{n_\rho}$. 

First let us consider the situation that multiple irreducible representations contribute, say $\rho_1,\dots,\rho_k$ where $k>1$. Consider the 1-parameter group $\{b_t : t \in \Rb\}$ where $b_t$ acts by $e^t$ on the $V_{\rho_1}^{n_{\rho_1}}$ factor and by the identity on all other factors. Then $b_t$ is not a scalar matrix, and centralizes $G$, so we are in case (1).

Therefore there is only one irreducible representation and $\Rb^{2p} \cong V_\rho^n$ for some irreducible representation $\rho$ and some $n$.

\begin{lemma}\label{lem:vcd} $n=1$. \end{lemma}

\begin{proof} Suppose for a contradiction that $n>1$. We first claim that $p=2$. Let us now consider the \emph{virtual cohomological dimension} $\vcd(\Gamma)$ of $\Gamma$. Recall that the \emph{cohomological dimension} $\cd(\Gamma)$ of $\Gamma$ is the supremum of all numbers $m$ such that $H^m(\Gamma,M)\neq 0$ for some $\Gamma$-module $M$ (see for instance~\cite[Chapter VIII]{B1994} for more information). We will only need the following properties of cd($\Gamma$):
	\begin{enumerate}
		\item $\cd(\Gamma)>0$ if $\Gamma\neq 1$. \label{item:nontriv}
		\item If $\Gamma$ acts freely and properly discontinuously on a contractible $CW$-complex $X$, then $\textrm{cd}(\Gamma)\leq \dim(X)$, with equality if and only if $X\slash\Gamma$ is compact. \label{item:geomdim}
		\item If $\Delta\subseteq \Gamma$, then $\cd(\Delta)\leq \cd(\Gamma)$. \label{item:mono}
		\item If $\Gamma=\Gamma_0\times\Gamma_1$ then $\cd(\Gamma)\leq \cd(\Gamma_0)+\cd(\Gamma_1)$. \label{item:subadd}
	\end{enumerate}
The virtual cohomological dimension of $\Gamma$ is then the infimum of $\cd(\Delta)$ as $\Delta$ ranges over finite index subgroups of $\Gamma$. 

Now write $\dim V_\rho = d$. Since $\Gamma_0$ can be identified with a discrete subgroup of $\PGL(V_\rho)$, we have by Property \ref{item:geomdim} above
	\begin{equation}
	\textrm{vcd}(\Gamma_0)\leq  \dim \SL_d(\Rb)\slash \SO(d) = \frac{d(d+1)}{2}-1.
	\label{eq:cd1}
	\end{equation}
Further, since $\Lambda$ commutes with $G^0$ and $\rho$ is an irreducible representation of $\wh{G}^0$, we can identify $\Lambda$ with a discrete subgroup of $\PGL_n(\Rb)$. Therefore 
	\begin{equation}
	\textrm{vcd}(\Lambda)\leq  \dim \SL_n(\Rb)\slash \SO(n) = \frac{n(n+1)}{2}-1.
	\label{eq:cd2}
	\end{equation}
On the other hand $\textrm{vcd}(\Gamma)=\dim\Omega=p^2$ by Property (2) above. Combining this with Property \ref{item:subadd} and Equations \ref{eq:cd1} and \ref{eq:cd2}, we have
	\begin{align*}
		2p^2 = 2\textrm{vcd}(\Gamma)&\leq 2\left(\textrm{vcd}(\Gamma_0)+\textrm{vcd}(\Lambda)\right)\\
							&\leq d(d+1)-2+n(n+1)-2\\
							&=d^2+d+n^2+n-4.
		\end{align*}
Using that $2p=dn$ (from the dimension count in $\Rb^{2p}\cong V_\rho^n$), we find that
	$$2p^2 \leq \frac{4p^2}{n^2}+\frac{2p}{n}+n^2+n-4.$$
The right-hand side is a convex function of $n$, so that on the interval $[2,p]$, it is maximal at one of the endpoints. At either endpoint the inequality reduces to 
	$$p^2-p-2\leq 0,$$
which is only possible if $p=2$. 

Then $(n,d) \in \{(2,2), (1,4), (4,1)\}$. We assumed that $n>1$ and since the representation $\wh{G}^0 \hookrightarrow \SL(V_\rho)$ is injective we must have $d >1$. So $n=d=2$. 

Thus $\wh{G}^0$ is a semisimple Lie group which has a faithful irreducible representation into $\SL_2(\Rb)$. Thus $\wh{G}^0$  has to be isomorphic to $\SL_2(\Rb)$ and $\rho=\Id$. With respect to the decomposition $\Rb^4 = V \oplus V$ we have 
\begin{align*}
\wh{G}^0 = \{ (\varphi, \varphi) \in \SL(V) \times \SL(V) \}.
\end{align*}

and hence we are in case (3) which is a contradiction. \end{proof}

Since $n=1$, $\wh{G}^0\curvearrowright \Rb^{2p}$ is an irreducible representation. Note that $\Lambda$ centralizes $G^0$ in $\PGL_{2p}(\Rb)$, and hence any element of $\GL_{2p}(\Rb)$ lying over $\Lambda$ has to be scalar by Schur's Lemma. It follows that $\Lambda$ is trivial, so that $G^\prime=G^0$ and thus $G^0$ has finite index in $G$. Then $\Gamma_0$ has finite index in $\Gamma$ and hence acts cocompactly on $\Omega$. Thus $\vcd(\Gamma_0) = \dim(\Omega)=p^2$.

\begin{lemma} $G^0$ acts transitively on $\Omega$. \end{lemma}

\begin{proof} Let $x\in\Omega$ be any point and let $K_x$ denote its stabilizer in $G^0$.Then $K_x$ is a compact subgroup of $G^0$ by Proposition~\ref{prop:proper} and the $G^0$-orbit $X$ of $x$ is is diffeomorphic to $G^0 \slash K_x$. Now let $K$ be a maximal compact subgroup of $G^0$ containing $K_x$. Then $\Gamma_0\backslash G^0\slash K$ is a closed aspherical manifold with fundamental group $\Gamma_0$ so by Property \ref{item:geomdim} of cohomological dimension we have $\textrm{vcd}(\Gamma_0)=\dim(G^0\slash K)$. On the other hand since $K_x \leq K$ and $G^0\slash K_x \cong X \subset \Omega$
	\begin{align*}
	\textrm{vcd}(\Gamma_0)&=\dim(G^0\slash K)\leq \dim(G^0\slash K_x)\\
						& = \dim(X) \leq \dim(\Omega)\\
						&=\textrm{vcd}(\Gamma_0).
	\end{align*}
We conclude that any $\dim(X)=\dim(\Omega)$, so that $X$ is a codimension 0 closed submanifold of $\Omega$. Connectedness of $\Omega$ then implies that $X=\Omega$, as desired. 
\end{proof}

\begin{remark}\label{rmk:max_cpct} The above proof shows that the stabilizer of any point $x \in \Omega$ has finite index in a maximal compact subgroup of $\Aut(\Omega)$. 
\end{remark}

\begin{lemma} $G^0$ is simple. \end{lemma}

\begin{proof}
Since $G^0$ has trivial center either $G^0$ is simple or $G^0 \cong G_1 \times G_2$ for some semisimple nontrivial Lie groups $G_1$ and $G_2$. 

So suppose that $G^0 \cong G_1\times G_2$. Let $\wh{G}_i$ be the inverse image of $G_i \times \{\Id\}$ under the map $\SL_{2p}(\Rb)\rightarrow \PSL_{2p}(\Rb)$. Next decompose the representation $\wh{G}_1 \curvearrowright \Rb^{2p}$ as a direct sum of irreducible representations of the semisimple group $\wh{G}_1$: 
	\begin{equation*}
	\Rb^{2p}\cong \bigoplus_{\tau} V_{\tau}^{n_\tau}.
	\end{equation*}
Here the direct sum is over nonisomorphic irreducible representations $\tau$ of $\wh{G}_1$, and $n_\tau \geq 0$ is the multiplicity of $\tau$. Using the fact that $\wh{G}_2$ centralizes $\wh{G}_1$ and arguing as in Lemma~\ref{lem:vcd} we see that $p=2$ and $\Rb^4 = V_{\tau}^2$ for some irreducible representation $\tau$ of $\wh{G}_1$. So $\dimension V_\tau = 2$ and thus $\wh{G}_1$ is isomorphic to $\SL_2(\Rb)$. Applying the same argument to $\wh{G}_2$ shows that $\wh{G}_2$ is also isomorphic to $\SL_2(\Rb)$. If we conjugate so that 
\begin{align*}
\wh{G}_1 = \left\{ \begin{pmatrix} A & 0 \\ 0 & A \end{pmatrix} : A \in \SL_2(\Rb) \right\}
\end{align*}
then an easy computation shows that the centralizer of $\wh{G}_1$ is exactly
\begin{align*}
\left\{ \begin{pmatrix} a\Id_2 & b\Id_2 \\ c\Id_2 & d\Id_2 \end{pmatrix} : ad-bc =1 \right\} \cong SL_2(\Rb).
\end{align*}
Since $\wh{G}_2$ centralizes $\wh{G}_1$ and is isomorphic to $SL_2(\Rb)$, we must have that 
\begin{align*}
\wh{G}_2=\left\{ \begin{pmatrix} a\Id_2 & b\Id_2 \\ c\Id_2 & d\Id_2 \end{pmatrix} : ad-bc =1 \right\}. 
\end{align*}
Hence we are in case (4), which is a contradiction.
\end{proof}

\section{The centralizer}\label{sec:center}

In this section we prove that case (1) in Theorem~\ref{thm:reduction} is impossible. For a subgroup $H \leq \PGL_{p+q}(\Rb)$ let
\begin{enumerate}
\item $\wh{H} = \{ h \in \GL_{p+q}(\Rb) : [h] \in H\}$,
\item $C_{H} = \{ c \in \End(\Rb^{p+q}) : c h = h c \text{ for all } h \in \wh{H} \}$, and
\item let $C_{H}^0$ be the connected component of $\Id_{p+q}$ in $C_{H} \cap \GL_{p+q}(\Rb)$. 
\end{enumerate}

With this notation we will prove the following:

\begin{theorem}\label{thm:no_center}
Suppose $\Omega \subset \Gr_p(\Rb^{2p})$ is an open set which is convex and bounded in some affine chart. If $\Gamma \leq \Aut(\Omega)$ is a discrete group that acts cocompactly on $\Omega$ then $C_\Gamma^0 = \Rb_{>0} \Id_{2p}$. 
\end{theorem}

\subsection{The centralizer in the general case}

We begin by proving the following (which holds for any Grassmannian):

\begin{theorem}\label{thm:cen_str}
Suppose $\Omega \subset \Gr_p(\Rb^{p+q})$ is an open $\Rc$-proper set that is convex in some affine chart. If $H \leq \Aut(\Omega)$ acts cocompactly on $\Omega$ then $C_{H}^0 \leq \Aut(\Omega)$ and there is a decomposition $\Rb^{p+q} = \oplus_{i=1}^m V_i$ so that 
\begin{align*}
C_H = \bigoplus_{i=1}^m \Rb \cdot \Id_{V_i}.
\end{align*}
\end{theorem}

\begin{remark} In the special case where $p=1$ the above Theorem is due to Vey~\cite[Theorem 5]{V1970}. \end{remark}

For the rest of this subsection assume that $\Omega \subset \Gr_p(\Rb^{p+q})$ and $H \leq \Aut(\Omega)$ satisfy the hypothesis of Theorem~\ref{thm:cen_str}.

\begin{lemma}\label{lem:C1} With the notation above, $C_{H}^0 \leq \Aut(\Omega)$
\end{lemma}

\begin{proof}
Fix a compact set $K \subset \Omega$ so that $H \cdot K = \Omega$. Then there exists a symmetric neighborhood $\Oc$ of $\Id_{p+q}$ in $C_{H}^0$ so that $\Oc$ generates $C_{H}^0$ and $u \cdot K \subset \Omega$ for all $u \in \Oc$. Then for $u \in \Oc$
\begin{align*}
u \cdot \Omega = u \cdot H \cdot K = H \cdot u \cdot K \subset H \cdot \Omega = \Omega.
\end{align*}
Since $\Oc$ is symmetric we also see that $u^{-1} \cdot \Omega \subset \Omega$. Thus $u$ restricts to a diffeomorphism $\Omega \rightarrow \Omega$ and $u \in \Aut(\Omega)$. Since $\Oc$ generates $C_{H}^0$ we then see that $C_{H}^0 \leq \Aut(\Omega)$. 
\end{proof}

\begin{lemma}\label{lem:finitetranslation}With the notation above, if $c \in C_{H}^0$ then 
\begin{align*}
\sup_{x \in \Omega} K_{\Omega}(c x, x) < \infty.
\end{align*}
\end{lemma}

\begin{proof}
Fix some $x_0 \in \Omega$. Then there exists $R>0$ so that 
\begin{align*}
\bigcup_{h \in H} B_R(h x_0) = \Omega.
\end{align*}
If $x \in \Omega$ pick $h \in H$ so that $K_{\Omega}(x, h x_0) \leq R$. Then 
\begin{align*}
K_{\Omega}(c x, x) 
&\leq K_{\Omega}(c x, c h x_0) + K_{\Omega}(c h x_0, h x_0) + K_{\Omega}(hx_0, x)\\
&\leq K_{\Omega}(x, h x_0) + K_{\Omega}(c x_0, x_0) + R\\
& \leq 2R + K_{\Omega}( c x_0, x_0).
\end{align*}
\end{proof}

\begin{lemma} With the notation above, if $c \in C_{H}^0$ then $c$ fixes every $\Rc$-extreme point of $\Omega$. 
\label{lem:centralfixes}
\end{lemma}

\begin{proof}
For an $\Rc$-extreme point $x \in \partial \Omega$, choose points $p_n\in\Omega$ with $p_n \rightarrow x$. By Lemma~\ref{lem:finitetranslation}, we have
\begin{align*}
\limsup_{n \rightarrow \infty} d_{\Omega} ( c p_n, p_n) < \infty.
\end{align*}
Then by Corollary~\ref{cor:asym_faces}, we have $c p_n \rightarrow x$. Since $c$ acts continuously on $\Gr_p(\Rb^{2p})$ and $p_n \rightarrow x$, we must have that $cx=x$.
\end{proof}

\begin{lemma}
With the notation above, every $c \in C_{H}^0$ is semisimple and $C_H^0$ is abelian. 
\end{lemma}

\begin{proof}
Fix a basis $v_1, \dots, v_D$ of $\wedge^p \Rb^{p+q}$ so that each $[v_i]$ is an $\Rc$-extreme point of $\Omega$ (this is possible by Proposition~\ref{prop:ext_spanning}). Then for any $c \in C_H^0$, each $v_i$ is an eigenvector of $\wedge^p c$ so $\wedge^p c$ is diagonalizable with respect to the basis $v_1, \dots, v_D$ of $\wedge^p \Rb^{p+q}$. Hence $\wedge^p C_H^0$ is an abelian group. 

Now $\wedge^p : \GL_{p+q}(\Rb) \rightarrow \GL(\wedge^p \Rb^{p+q})$ is an injective homomorphism, maps unipotents to unipotents, and maps semisimple elements to semisimple elements. Since $\wedge^p : \GL_{p+q}(\Rb) \rightarrow \GL(\wedge^p \Rb^{p+q})$ is injective we see that $C_H^0$ is abelian. 

We next claim that any $c \in C_H^0$ is semisimple. If $c = s u$ is the Jordan decomposition of $c$ then $\wedge^p c = (\wedge^p s )(\wedge^p u)$ and by uniqueness this is the Jordan decomposition of $\wedge^p c$. It follows that $\wedge^p u = 1$, and hence $u=1$. We conclude that $c=s$ is semisimple.
\end{proof}

\begin{lemma}
With the notation above, there is a decomposition $\Rb^{p+q} = \oplus_{i=1}^m V_i$ so that 
\begin{align*}
C_H = \bigoplus_{i=1}^m \Rb \cdot \Id_{V_i}.
\end{align*}
\end{lemma}

\begin{proof} This follows from parts (2)-(4) of the proof of Theorem 5 in~\cite{V1970}.
\end{proof}

\subsection{The centralizer in $\Gr_p(\Rb^{2p})$}

We now specialize to the case in which $p=q$ and prove Theorem~\ref{thm:no_center}. We begin by showing that we can assume that $\Omega$ is a cone in some affine chart. 

\begin{proposition}\label{prop:norm_center}
Suppose $\Omega \subset \Gr_p(\Rb^{p+q})$ is an open set which is convex and bounded in some affine chart. If $H \leq \Aut(\Omega)$ acts cocompactly on $\Omega$ and $C_{H}^0 \neq \Rb_{>0} \Id_{2p}$, then there exists $\varphi \in \GL_{2p}(\Rb)$ so that 
\begin{align*}
\varphi \Omega \subset  \Mb=\left\{\begin{bmatrix} \Id_p \\ X \end{bmatrix}: X \in M_{p,p}(\Rb)\right\}
\end{align*}
and  $\varphi\Omega$ is a convex cone in $\Mb$ based at 0. Moreover, either 
\begin{align*}
C_{\varphi H \varphi^{-1}}^0 = \left\{ \begin{pmatrix} e^t\Id_{p} & 0 \\ 0 & e^{s} \Id_{p} \end{pmatrix} : s,t \in \Rb\right\}.
\end{align*}
or $C_{\varphi H \varphi^{-1}}^0$ contains the subgroup
\begin{align*}
\left\{ \begin{pmatrix} e^t\Id_{p+\ell} & 0 \\ 0 & e^{s} \Id_{p-\ell} \end{pmatrix} : s,t \in \Rb\right\} 
\end{align*}
for some $0 < \ell < p$.
\end{proposition}

\begin{proof}
We can assume that $\Omega$ is a convex bounded subset of $\Mb$. Through out the argument we will replace $\Omega$ by translates of the form 
\begin{align*}
\begin{bmatrix} A & 0 \\ B & C \end{bmatrix} \Omega
\end{align*}
this transformation preserves the affine chart $\Mb$ and acts by affine transformations. 

By Theorem~\ref{thm:cen_str}, there exists $g_0 \in \GL_{2p}(\Rb)$ and $0 \leq \ell < p$ so that 
\begin{align*}
A:= \left\{ g_0 \begin{pmatrix} e^t \Id_{p+\ell} & 0 \\ 0 & e^s \Id_{p-\ell} \end{pmatrix}g_0^{-1} : s,t \in \Rb \right\} \leq C_H^0.
\end{align*}
Notice that we can choose $\ell > 0$ except when 
\begin{align*}
C_H^0 = \left\{ g_0 \begin{pmatrix} e^t \Id_{p} & 0 \\ 0 & e^s \Id_{p} \end{pmatrix}g_0^{-1} : s,t \in \Rb \right\}.
\end{align*}
Now let $W:= g_0 \Spanset \{ e_1, \dots, e_{p+\ell}\}$. Notice that $hW = W$ for all $h \in H$. We claim that there exists an $\Rc$-extreme point $e$ of $\Omega$ in $\Gr_p(W)$. Consider some
\begin{align*}
c = g_0 \begin{pmatrix} e^t \Id_{p+\ell} & 0 \\ 0 & e^s \Id_{p-\ell} \end{pmatrix}g_0^{-1} \in A
\end{align*}
with $e^t > e^s$. Then $E^+(\wedge^p c) \cap Gr_p(\Rb^{2p}) \subset \Gr_p(W)$ and by Corollary~\ref{cor:ext_attracting} there is an $\Rc$-extreme point $e$ of $\Omega$ in $E^+(\wedge^p c) \cap \partial \Omega$. 

Now by replacing $\Omega$ with an affine translate we can assume that 
\begin{align*}
e = \begin{bmatrix} \Id_p \\ 0 \end{bmatrix}
\end{align*}
which implies that $\Spanset\{ e_1, \dots, e_p\} \subset W$. Notice that this implies that any $a \in A$ can be written as 
\begin{align*}
a = \begin{pmatrix} e^t \Id_p & B \\ 0 & C \end{pmatrix}
\end{align*}
for some $t \in \Rb$ and $B, C \in \GL_p(\Rb)$.

Since $e$ is an extreme point, by Corollary~\ref{cor:rescaling}, there exists $t_n \rightarrow \infty$ and $h_n \in H$ so that 
\begin{align*}
\varphi = \lim_{n \rightarrow \infty} \begin{bmatrix} \Id_p & 0 \\ 0 & e^{t_n} \Id_p \end{bmatrix} h_n 
\end{align*}
in $\PGL_{2p}(\Rb)$ and $\varphi(\Omega) = \Tc \Cc_0 \Omega$. Let $\wh{\varphi} \in \GL_{2p}(\Rb)$ be a representative of $\varphi$ and for each $n \in \Nb$ pick $\wh{h}_n \in \GL_{2p}(\Rb)$ a representative of $h_n$ so that 
\begin{align*}
\wh{\varphi} =  \lim_{n \rightarrow \infty} \begin{pmatrix} \Id_p & 0 \\ 0 & e^{t_n} \Id_p \end{pmatrix} \wh{h}_n
\end{align*}
in $\GL_{2p}(\Rb)$. 

Then if 
\begin{align*}
a = \begin{pmatrix} e^t \Id_p & B \\ 0 & C \end{pmatrix} \in A
\end{align*}
we have
\begin{align*}
\wh{\varphi} a \wh{\varphi}^{-1} 
&= \lim_{n \rightarrow \infty} \begin{pmatrix} \Id_p & 0 \\ 0 & e^{t_n} \Id_p \end{pmatrix} \wh{g}_n  \begin{pmatrix} e^t \Id_p & B \\ 0 & C \end{pmatrix}  \wh{g}_n ^{-1} \begin{pmatrix} \Id_p & 0 \\ 0 & e^{-t_n} \Id_p \end{pmatrix} \\
& = \lim_{n \rightarrow \infty} \begin{pmatrix} \Id_p & 0 \\ 0 & e^{t_n} \Id_p \end{pmatrix}\begin{pmatrix} e^t \Id_p & B \\ 0 & C \end{pmatrix}  
  \begin{pmatrix} \Id_p & 0 \\ 0 & e^{-t_n} \Id_p \end{pmatrix} \\
  & =  \begin{pmatrix} e^t \Id_p & 0 \\ 0 & C \end{pmatrix}.
\end{align*}

So replacing $\Omega$ with a affine translate we can assume 
\begin{align*}
\wh{\varphi} A \wh{\varphi}^{-1}=
\left\{ \begin{pmatrix} e^t\Id_{p+\ell} & 0 \\ 0 & e^{s} \Id_{p-\ell} \end{pmatrix} : s,t \in \Rb\right\}.
\end{align*}
This completes the proof.
\end{proof}

\begin{proof}[Proof of Theorem~\ref{thm:no_center}]
By Proposition \ref{prop:norm_center}, we can assume that 
\begin{align*}
 \Omega \subset  \Mb=\left\{\begin{bmatrix} \Id_p \\ X \end{bmatrix}: X \in M_{p,p}(\Rb)\right\}
\end{align*}
is a convex cone in $\Mb$ based at 0, and that $C_{\Gamma}^0$ contains the subgroup 
\begin{align*}
\left\{ \begin{pmatrix} e^t\Id_{p+\ell} & 0 \\ 0 & e^{s} \Id_{p-\ell} \end{pmatrix} : s,t \in \Rb\right\} 
\end{align*}
for some $0 \leq \ell < p$.

Then 
\begin{align*}
\Gamma \leq \left\{ \begin{bmatrix} A & 0 \\ 0 & B \end{bmatrix} : A \in \GL_{p+\ell}(\Rb), B \in \GL_{p-\ell}(\Rb) \right\}.
\end{align*}

Throughout the argument we will write a matrix $X \in M_{p,p}(\Rb)$ as 
\begin{align*}
X = \begin{pmatrix} X_1 \\ X_2 \end{pmatrix}
\end{align*}
where $X_1 \in M_{\ell,p}(\Rb)$ and $X_2 \in M_{p-\ell,p}(\Rb)$. Let 
\begin{align*}
\Omega_2 = \left\{ \begin{bmatrix} \Id_p \\ 0  \\ X_2 \end{bmatrix} : \text{ there exists } X_1 \text{ so that } \begin{bmatrix} \Id_p \\ X_1 \\ X_2 \end{bmatrix}  \in \Omega \right\}.
\end{align*}

\begin{lemma} $\Omega_2$ is a proper convex cone in $\Mb$, that is $\Omega_2$ does not contain any affine lines. \end{lemma}

\begin{proof} 
Since
\begin{align*}
\{ x+ tv : t \in \Rb\} \subset \Omega_2 & \Leftrightarrow \{ x^\prime + tv : t \in \Rb\} \subset \Omega_2 \text{ for all } x^\prime \in \Omega_2 \\
& \Leftrightarrow  \begin{bmatrix} \Id_p & 0 & 0 \\ 0 & \Id_\ell & 0 \\ v & 0 & \Id_{p-\ell} \end{bmatrix} \in \Aut(\Omega). 
\end{align*}
it suffices to show that 
\begin{align*}
\{ \Id_{2p}\} = \left\{ \begin{bmatrix} \Id_{p+\ell} & 0 \\ Y & \Id_{p-\ell} \end{bmatrix} : Y \in M_{p-\ell, p+\ell}(\Rb) \right\} \cap \Aut(\Omega).
\end{align*}
So suppose that
\begin{align*}
g:=\begin{bmatrix} \Id_{p+\ell} & 0 \\ Y & \Id_{p-\ell} \end{bmatrix} \in \Aut(\Omega)
\end{align*}
for some $Y\in M_{p-\ell,p+\ell}(\Rb)$. Since $\Gamma$ is a cocompact lattice in $\Aut(\Omega)$, there exist
\begin{align*}
\gamma_n := \begin{bmatrix} A_n & 0 \\ 0 & B_n \end{bmatrix} \in \Gamma
\end{align*}
such that $\{\gamma_n\ g_n\}_n$ is bounded in $\PGL_{2p}(\Rb)$. By picking representatives of $\gamma_n$ and $g_n$ in $\GL_{2p}(\Rb)$ correctly we can assume that 
 \begin{align*}
\begin{pmatrix} A_n & 0 \\ 0 & B_n \end{pmatrix}\begin{pmatrix} \Id_{p+\ell} & 0 \\ n Y & \Id_{p-\ell} \end{pmatrix} = \begin{pmatrix} A_n & 0 \\ n B_nY & B_n \end{pmatrix}
\end{align*}
is a bounded sequence in $\GL_{2p}(\Rb)$. This implies $\{B_n\}_n$ and $\{n B_n Y\}_n$ are bounded sequences in $\GL_{p-\ell}(\Rb)$ and $M_{p-\ell,p+\ell}(\Rb)$ respectively. Therefore we must have $Y=0$, as desired.
\end{proof}

Since Proposition~\ref{prop:norm_center} yields different conclusions as to whether $\ell=0$ or $\ell>0$, we will consider these two situations separately below. \newline

\noindent \textbf{Case 1:} First suppose that $\ell=0$. Then $\Omega = \Omega_2$ is a proper convex cone and by Proposition~\ref{prop:norm_center} we may assume that
\begin{align*}
C_{\Gamma}^0 = \left\{ \begin{pmatrix} e^t\Id_{p} & 0 \\ 0 & e^{s} \Id_{p} \end{pmatrix} : s,t \in \Rb\right\}.
\end{align*}
Then 
\begin{align*}
\Gamma \leq \left\{ \begin{bmatrix} A & 0 \\ 0 & B \end{bmatrix} : A, B \in \GL_p(\Rb)\right\}.
\end{align*}
So $\Gamma$ acts by linear transformations on $\Omega$. We will now use theory of linear automorphisms of a proper convex cone to establish a contradiction. 

Define a homomorphism 
\begin{align*}
\Phi: \left\{ \begin{bmatrix} A & 0 \\ 0 & B \end{bmatrix} \in \PGL_{2p}(\Rb) : A, B \in \GL_p(\Rb) \right\} \rightarrow \GL(\Mb)
\end{align*}
by
\begin{align*} 
\Phi \left( \begin{bmatrix} A & 0 \\ 0 & B \end{bmatrix}\right) \cdot X = BXA^{-1}.
\end{align*}
Notice that $\Phi$ is injective and well defined.

Then $\Lambda:=\Phi(\Gamma)$ acts cocompactly on $\Omega \subset \Mb$. Let $ \overline{\Gamma}^Z$ be the Zariski closure of $\Gamma$ in $\PGL_{2p}(\Rb)$ and $ \overline{\Lambda}^{Z}$ the Zariski closure of $\Lambda$ in $\GL(\Mb)$. Then
\begin{align*}
\Phi \left( \overline{\Gamma}^Z\right) = \overline{\Lambda}^{Z}.
\end{align*}
By possibly passing to a finite index subgroup we can assume that $\overline{\Gamma}^Z$ is connected in the Zariski topology. 

Let $C_\Lambda \leq \GL(\Mb)$ denote the centralizer of $\Lambda$ in $\GL(\Mb)$. By a result of Vey~\cite[Theorem 5]{V1970} either $\Omega$ is an irreducible cone and $C_\Lambda =\Rb^* \cdot \Id_{\Mb}$ or $\dim C_{\Lambda} > 1$. By~\cite[Theorem 1.1]{B2003}, we see that $C_\Lambda \leq \overline{\Lambda}^Z$. Now if $[C_\Gamma^0]$ is the image of $C_\Gamma^0$ in $\PGL_{2p}(\Rb)$ we see that
\begin{align*}
\Phi^{-1} ( C_{\Lambda}) \subset [C_{\Gamma}^0].
\end{align*}
Since $\dim [C_{\Gamma}^0] = 1$, so we see that $\dim C_{\Lambda} =1$. Thus $\Omega$ is an irreducible cone. Then by~\cite[Theorem 3]{V1970}  (see also~\cite{B2003}) there exists a simple group $H \leq \GL(\Mb)$ so that 
\begin{align*}
\overline{\Lambda}^{Z}  = (\Rb^* \Id) \cdot H.
\end{align*}
 So $\overline{\Gamma}^{Z} \cong \Rb^* \times H$.
 
Now consider the projections
\begin{align*}
\pi_1, \pi_2 : \overline{\Gamma}^Z \rightarrow \PGL_p(\Rb)
\end{align*}
given by 
\begin{align*}
\pi_1 \left( \begin{bmatrix} A & 0 \\ 0 & B \end{bmatrix}\right) = A \text{ and } \pi_2 \left( \begin{bmatrix} A & 0 \\ 0 & B \end{bmatrix}\right) = B.
\end{align*}
Since $H$ is simple, we see that $\ker \pi_i = \overline{\Gamma}^Z$ or 
\begin{align*}
\ker \pi_i =  \left\{ \begin{bmatrix} e^t\Id_{p} & 0 \\ 0 & e^{s} \Id_{p} \end{bmatrix} \in \PGL_{2p}(\Rb) : s,t \in \Rb\right\}.
\end{align*}
Since 
\begin{align*}
\ker (\pi_1 \times \pi_2) =  \left\{ \begin{bmatrix} e^t\Id_{p} & 0 \\ 0 & e^{s} \Id_{p} \end{bmatrix}  \in \PGL_{2p}(\Rb) : s,t \in \Rb\right\}
\end{align*}
we must have that $\ker \pi_i \neq \overline{\Gamma}^Z$ for some $i \in \{1,2\}$ . Then we see that 
\begin{align*}
\pi_i \circ \Phi^{-1} : H \rightarrow \PGL_p(\Rb)
\end{align*}
 is an injection and thus we obtain an injective homomorphism
\begin{align*}
\overline{\Gamma}^Z \hookrightarrow \Rb \times \PGL_p(\Rb).
\end{align*}
But then 
\begin{align*}
p^2 = \dim(\Omega) = \vcd(\Gamma) \leq 1 + \dim ( \SL_p(\Rb) / \SO(p) ) =  \frac{p(p+1)}{2} = \frac{1}{2} p^2 + \frac{1}{2} p < p^2
\end{align*}
which is a contradiction.  \newline

\noindent \textbf{Case 2:}  Suppose that $C_{\Gamma}^0$ contains the subgroup 
\begin{align*}
\left\{ \begin{pmatrix} e^t\Id_{p+\ell} & 0 \\ 0 & e^{s} \Id_{p-\ell} \end{pmatrix} : s,t \in \Rb\right\} 
\end{align*}
for some $0 < \ell < p$. 

Let
\begin{align*}
\Omega_1 = \left\{ \begin{bmatrix} \Id_p \\ X_1 \\ 0 \end{bmatrix} : \text{ there exists } X_2 \text{ so that } \begin{bmatrix} \Id_p \\ X_1 \\ X_2 \end{bmatrix}  \in \Omega \right\}.
\end{align*}

\begin{lemma} $\Omega = \Omega_1 + \Omega_2$. \end{lemma}
\begin{proof}
By construction
\begin{align*}
\overline{\Omega} \subset \overline{\Omega_1} + \overline{\Omega_2}.
\end{align*}

Now 
\begin{align*}
\begin{pmatrix} \Id_{p+\ell} & 0 \\ 0 & e^{s}\Id_{p-\ell} \end{pmatrix} \cdot \begin{bmatrix} \Id_p \\ X_1 \\ X_2 \end{bmatrix} = \begin{bmatrix} \Id_p \\ X_1 \\ e^s X_2 \end{bmatrix}. 
\end{align*}
So by sending $s\rightarrow -\infty$ we see that 
\begin{align*}
\overline{\Omega} \supset \overline{\Omega_1}.
\end{align*}
On the other hand, 
\begin{align*}
\begin{pmatrix} \Id_{p} & 0 \\ 0 & e^{-s}\Id_{p} \end{pmatrix} \begin{pmatrix} \Id_{p+\ell} & 0 \\ 0 & e^{s}\Id_{p-\ell} \end{pmatrix} \cdot \begin{bmatrix} \Id_p \\ X_1 \\ X_2 \end{bmatrix} = \begin{bmatrix} \Id_p \\ e^{-s}X_1 \\  X_2 \end{bmatrix}.
\end{align*}
So sending $s \rightarrow \infty$ we see that 
\begin{align*}
\overline{\Omega} \supset \overline{\Omega_2}.
\end{align*}
Then if $X_1 \in \overline{\Omega_1}$ and $X_2 \in \overline{\Omega_2}$ we have 
\begin{align*}
X_1 + X_2 = \frac{1}{2}(2X_1) + \frac{1}{2}(2X_2) \in \overline{\Omega}.
\end{align*}
Thus $\overline{\Omega} = \overline{\Omega_1} + \overline{\Omega_2}$ which by convexity implies that 
\begin{align*}
\Omega = \Omega_1 + \Omega_2.
\end{align*} 
\end{proof}

Now if $\gamma \in \Gamma$ then we can write
\begin{align*}
\gamma = \begin{bmatrix} A_1 & A_2 & 0 \\ A_3 & A_4 & 0 \\ 0 & 0 & B \end{bmatrix}
\end{align*}
for some $A_1 \in M_{p,p}(\Rb)$, $A_2 \in M_{p, \ell}(\Rb)$, $A_3 \in M_{\ell, p}(\Rb)$, $A_4 \in M_{\ell, \ell}(\Rb)$, and $B \in \GL_{p-\ell}(\Rb)$. With this decomposition 
\begin{align*}
\begin{bmatrix} A_1 & A_2 & 0 \\ A_3 & A_4 & 0 \\ 0 & 0 & B \end{bmatrix} \cdot \begin{bmatrix} \Id_p \\ X_1 \\ X_2 \end{bmatrix} = \begin{bmatrix} \Id_p \\ (A_3+A_4X_1)(A_1+A_2X_1)^{-1} \\ BX_2(A_1+A_2X_1)^{-1} \end{bmatrix}.
\end{align*}

Now by identifying $M_{p-\ell, p}(\Rb)$ with $\Rb^{(p-\ell)p}$ we can view $\Omega_2$ as a convex subset of $\Pb(\Rb^{(p-\ell)p+1})$. Let $e$ be an extreme point of $\Omega_2$ in $\Pb(\Rb^{(p-\ell)p+1}) \setminus \Rb^{(p-\ell)p+1}$. Fix a sequence of points $y_n \in \Omega_2$ which converges to $e$ in $\Pb(\Rb^{(p-\ell)p+1})$

Next fix some $x_0 \in \Omega_1$ and consider the sequence
\begin{align*}
z_n = \begin{bmatrix} \Id_p \\ x_0 \\ y_n \end{bmatrix} \in \Omega.
\end{align*}
Then there exists $\gamma_n \in \Gamma$ and a compact subset $K$ of $\Omega$ so that 
\begin{align*}
\gamma_n^{-1} z_n \in K.
\end{align*}
Suppose  
\begin{align*}
\gamma_n = \left[\begin{pmatrix} A_1^{(n)} & A_2^{(n)} & 0 \\ A_3^{(n)} & A_4^{(n)} & 0 \\ 0 & 0 & B^{(n)} \end{pmatrix}\right].
\end{align*}

Now let 
\begin{align*}
\GL(\Omega_2) = \{ T \in \GL(M_{p-\ell,p}(\Rb)) : T(\Omega_2) = \Omega_2 \}.
\end{align*}
Since $\Omega_2 \subset M_{p-\ell,p}(\Rb)$ is a proper convex cone, the Hilbert metric $H_{\Omega_2}$ is a complete $\GL(\Omega_2)$-invariant metric on $\Omega_2$. Moreover, since $\Omega = \Omega_1 + \Omega_2$ we see that the linear map 
\begin{align*}
T_n(X)= B^{(n)} X ( A_1^{(n)} + A_2^{(n)} x_0)^{-1}
\end{align*}
is in $\GL(\Omega_2)$ for all $n \geq 0$. So there exists $R \geq 0$ so that 
\begin{align*}
H_{\Omega_2} ( y_n, B^{(n)} y_0 ( A_1^{(n)} + A_2^{(n)} x_0)^{-1}) \leq R
\end{align*}
for all $n \geq 0$. Since $y_n$ converges to an extreme point of $\Omega_2$ we see that $[T_n] \in \Pb(\End( M_{p-\ell,p}(\Rb)))$ converges to some $T_\infty \in \Pb(\End( M_{p-\ell,p}(\Rb)))$ and $\rank T_\infty =1$ (see either Vey~\cite[Lemma 4]{V1970} or Theorem~\ref{thm:extreme} above).

Now if $\sigma_1^{(n)} \geq \dots \geq \sigma_{p-\ell}^{(n)}$ are the singular values of $B^{(n)}$ and $\mu_1^{(n)} \geq \dots \geq \mu_{p}^{(n)}$ are the singular values of $( A_1^{(n)} + A_2^{(n)} x_0)^{-1}$ then $T_n$ has singular values
\begin{align*}
\{ \sigma_i^{(n)} \mu_j^{(n)} : 1 \leq i \leq p-\ell, 1 \leq j \leq p\}.
\end{align*}
Then since $[T_n] \rightarrow T_\infty$ and $\rank T_\infty = 1$ we must have
\begin{align*}
\lim_{n \rightarrow \infty} \frac{ \sigma_1^{(n)} \mu_1^{(n)}}{ \sigma_i^{(n)} \mu_j^{(n)}} = \infty
\end{align*}
for all $1 \leq i \leq p-\ell, 1 \leq j \leq p$ with $(i,j) \neq (1,1)$. 

In particular, 
\begin{align*}
\lim_{n \rightarrow \infty}  \mu_1^{(n)}/  \mu_2^{(n)} = \infty.
\end{align*}

So we will finish the proof by establishing the following:

\begin{lemma}
\begin{align}
\label{eq:limit}
\limsup_{n \rightarrow \infty}  \mu_1^{(n)}/  \mu_2^{(n)} < \infty
\end{align}
\end{lemma}

\begin{proof}
Now view $\Omega_1$ as an open subset of $\Gr_p( V)$ where $V = \Spanset\{ e_1, \dots, e_{p+\ell}\}$. By construction $\Omega_1$ is an $\Rc$-proper convex open subset of some affine chart of $\Gr_p(V)$. Thus $K_{\Omega_1}$ is a proper metric and there exists $R_1 \geq 0$ so that 
\begin{align*}
K_{\Omega_1}(x_0, (A_3^{(n)} + A_4^{(n)}x_0)(A_1^{(n)}+A_2^{(n)}x_0)^{-1} ) \leq R_1.
\end{align*}
So the set 
\begin{align*}
\left\{ \begin{bmatrix} A_1^{(n)} & A_2^{(n)}  \\ A_3^{(n)} & A_4^{(n)} \end{bmatrix} : n \in \Nb\right\} \subset \PGL(V)
\end{align*}
is relatively compact in $\PGL(V)$. So we can pass to a subsequence and pick representatives so that 
\begin{align*}
\begin{pmatrix} A_1^{(n)} & A_2^{(n)}  \\ A_3^{(n)} & A_4^{(n)} \end{pmatrix} \rightarrow \begin{pmatrix} A_1 & A_2  \\ A_3 & A_4 \end{pmatrix}
\end{align*}
in $\GL(V)$. Now we claim that $(A_1 + A_2x_0)$ is an invertible matrix. Suppose not then for each $n$ we can find an unit eigenvector $v_n \in \Cb^p$ so that 
\begin{align*}
(A_1^{(n)}+A_2^{(n)}x_0)v_n \rightarrow 0
\end{align*} 
Since $(A_3^{(n)} + A_4^{(n)}x_0)(A_1^{(n)}+A_2^{(n)}x_0)^{-1}$ stays within a compact subset of $\Omega_2$ we then must have that $(A_3^{(n)} + A_4^{(n)}x_0)v_n \rightarrow 0$. But then we can pass to a subsequence so that $v_n \rightarrow v$ and then 
\begin{align*}
0 = \lim_{n \rightarrow \infty} \begin{pmatrix} A_1^{(n)} & A_2^{(n)}  \\ A_3^{(n)} & A_4^{(n)} \end{pmatrix} \begin{pmatrix} v_n \\ x_0 v_n\end{pmatrix} = \begin{pmatrix} A_1 & A_2  \\ A_3 & A_4 \end{pmatrix} \begin{pmatrix} v \\ q_0 v\end{pmatrix}
\end{align*}
which contradicts the fact that 
\begin{align*}
\begin{pmatrix} A_1 & A_2  \\ A_3 & A_4 \end{pmatrix} \in \GL_{p+\ell}(\Rb).
\end{align*}
So $(A_1 + A_2q_0)$  is an invertible matrix. But this implies that there exists an $C >0$ so that 
\begin{align*}
\{ \mu^{(n)}_i : 1 \leq i \leq p\} \subset [1/C, C]
\end{align*}
which contradicts Equation~\ref{eq:limit}. 
\end{proof}
\end{proof}

\section{Unipotent subgroups}\label{sec:unipotent}

In this section we show that case (2) of Theorem~\ref{thm:reduction} is impossible. 

\begin{theorem}\label{thm:no_unipotent}
Suppose $\Omega \subset \Gr_p(\Rb^{2p})$ is an open set which is bounded and convex in some affine chart. If $\Gamma \leq \Aut(\Omega)$ is a discrete group which acts cocompactly on $\Omega$ then there does not exists a nontrivial abelian normal unipotent group $U \leq \Aut(\Omega)$ such that $\Gamma \cap U$ is a cocompact lattice in $U$.
\end{theorem}

For the rest of the section suppose $\Omega \subset \Gr_p(\Rb^{2p})$ and $\Gamma \leq \Aut(\Omega)$ satisfy the hypothesis of Theorem~\ref{thm:no_unipotent}. Assume for a contradiction that there exists a nontrivial abelian normal unipotent group $U \leq \Aut(\Omega)$ such that $\Gamma \cap U$ is a cocompact lattice in $U$.

Since $\Gamma$ is finitely generated, by passing to a finite index subgroup we can assume that $\Gamma$ is torsion free. Then since $\Gamma$ acts properly on $\Omega$ we see that $\Gamma$ acts freely on $\Omega$. Then, using the fact that $\Gamma \backslash \Omega$ is compact, we see that 
\begin{align}
\label{eq:trans_length}
\inf_{\gamma \in \Gamma, x \in \Omega} K_{\Omega}(\gamma x, x) > 0.
\end{align}
The basic idea of the following argument is that if $u \in U \cap\Gamma$ then the translation distance 
\begin{align*}
\inf_{x \in \Omega} K_{\Omega}(u x, x)
\end{align*}
should be zero which contradicts the fact that $U \cap \Gamma \neq \emptyset$. 

The group $\wedge^p U \leq \PGL(\wedge^p \Rb^{2p})$ is also unipotent so the set
\begin{align*}
E_1 = \{ v \in \Pb(\wedge^p \Rb^{2p}) : (\wedge^p u) v = v \text{ for all } u \in U\}
\end{align*}
is non-empty. Moreover, there exists some $u_0 \in U \cap \Gamma$ so that
\begin{align*}
E_1 = \{ v \in \Pb(\wedge^p \Rb^{2p}) : (\wedge^p u_0) v = v\}.
\end{align*}
Then with the notation of Proposition~\ref{prop:attracting} 
\begin{align*}
E^+(\wedge^p u_0) \subset E_1
\end{align*}
and by Corollary~\ref{cor:ext_attracting} there exists an $\Rc$-extreme point $e \in E^+(\wedge^p u_0) \cap \partial \Omega$. 

Now suppose that $\Omega$ is a bounded convex open set in the affine chart 
\begin{align*}
\Mb = \left\{ \begin{bmatrix} \Id_p \\ X \end{bmatrix} : X \in M_{p,p}(\Rb) \right\}.
\end{align*}
Without loss of generality we can assume $e = 0$ in this affine chart. Then by Corollary~\ref{cor:rescaling}, there exists $\gamma_n \in \Gamma$ and $t_n \rightarrow \infty$ so that 
\begin{align*}
\varphi = \lim_{n \rightarrow \infty} \begin{bmatrix} \Id_p & 0 \\ 0 & e^{t_n} \Id_p \end{bmatrix} \gamma_n \in \PGL_{2p}(\Rb)
\end{align*}
and $\varphi\Omega \subset \Mb$ is a $\Rc$-proper convex open cone based at $0$. In particular, $\Aut(\varphi \Omega)$ contains the one parameter subgroup 
\begin{align*}
a_t := \begin{bmatrix} \Id_p & 0 \\ 0 & e^{t} \Id_p \end{bmatrix}.
\end{align*}

Now if 
\begin{align*}
\varphi_n:= \begin{bmatrix} \Id_p & 0 \\ 0 & e^{t_n} \Id_p \end{bmatrix} \gamma_n
\end{align*}
then 
\begin{align*}
\varphi_n^{-1}(e) = \gamma_n^{-1}(e) \in \gamma_n^{-1} E_1 \cap \gamma_n^{-1}E^+(\wedge^p u_0) = E_1 \cap E^+( \wedge^p \gamma_n^{-1} u_0 \gamma_n)
\end{align*}
so 
\begin{align*}
\varphi_n^{-1}(e) \in E_1 \cap \Big(\cup_{u \in U } E^+(\wedge^p u)\Big)
\end{align*}
so sending $n \rightarrow \infty$ we see that 
\begin{align*}
\varphi^{-1}(e) \in E_1 \cap \overline{\cup_{u \in U } E^+(\wedge^p u)}.
\end{align*}
And thus
\begin{align*}
e \in \varphi(E_1) \cap \overline{ \cup_{u \in \varphi U \varphi^{-1} } E^+(\wedge^p u) }.
\end{align*}
In particular, since $e = \Spanset\{e_1, \dots, e_p\} \subset \varphi(E_1)$, we have
\begin{align*}
\varphi U \varphi^{-1} \leq \left\{ \begin{bmatrix} A & B \\ 0 & C \end{bmatrix} : A,B,C \in M_{p,p}(\Rb) \right\}.
\end{align*}

\begin{lemma} If 
\begin{align*}
\begin{bmatrix} \Id_p & X \\ 0 & \Id_p \end{bmatrix} \in \varphi U \varphi^{-1} 
\end{align*}
then $X=0$. 
\end{lemma}

\begin{proof}
Suppose for a contradiction that there exists $u = \begin{bmatrix} \Id_p & X \\ 0 & \Id_p \end{bmatrix} \in  \varphi U \varphi^{-1} $ and $X \neq 0$. Then since $\Gamma \cap U \leq U$ is a lattice and $U$ is a abelian there exists $n_k \rightarrow \infty$ and $\gamma_k \in \varphi(\Gamma \cap U) \varphi^{-1}$ such that $\gamma_k^{-1} u^{n_k} \rightarrow \Id_{2p}$. By picking representatives correctly we can assume that 
\begin{align*}
\gamma_k = \begin{bmatrix} A_k & B_k \\ 0 & C_k \end{bmatrix}
\end{align*}
and 
\begin{align*}
\begin{pmatrix}
A_k^{-1} & -A_k^{-1} B_k \\  0 & C_k^{-1}
\end{pmatrix}
\begin{pmatrix}
\Id_p & n_k X \\
0 & \Id_p 
\end{pmatrix} 
=
\begin{pmatrix}
A_k^{-1} & n_kA_k^{-1}X - A_k^{-1}B_k \\
0 & C_k^{-1}
\end{pmatrix}
\rightarrow 
\begin{pmatrix}
\Id_p & 0 \\
0 & \Id_p
\end{pmatrix}
\end{align*}
in $\GL_{2p}(\Rb)$. So $A_k \rightarrow \Id_p$ and $C_k \rightarrow \Id_p$. But then there exists $t_k \rightarrow \infty$ so that $a_{t_k} \gamma_k a_{-t_k} \rightarrow \Id_{2p}$. But then for any $p \in \varphi\Omega$
\begin{align*}
\lim_{k \rightarrow \infty} K_{\varphi \Omega}( \gamma_k a_{-t_k} p, a_{-t_k} p) = \lim_{k \rightarrow \infty} K_{\varphi \Omega}( a_{t_k} \gamma_k a_{-t_k} p, p) =0
\end{align*}
which contradicts Equation~\ref{eq:trans_length}.
\end{proof}

\begin{lemma}
\begin{align*}
\varphi U \varphi^{-1} \leq \left\{ \begin{bmatrix} A & 0 \\ 0 & B \end{bmatrix} : A,B \in \GL_p(\Rb) \right\}.
\end{align*}
\end{lemma}

\begin{proof}
Suppose for a contradiction that there exists 
\begin{align*}
u = \begin{bmatrix} A & B \\ 0 & C \end{bmatrix} \in \varphi U \varphi^{-1} 
\end{align*}
and $B \neq 0$. 

Then 
\begin{align*}
u^\prime = \begin{bmatrix} A & 0 \\ 0 & C \end{bmatrix} = \lim_{t \rightarrow \infty} a_t u a_{-t} \in  \varphi U \varphi^{-1} 
\end{align*}
and so 
\begin{align*}
\begin{bmatrix} \Id_p & A^{-1}B \\ 0 & \Id_p \end{bmatrix} = (u^\prime)^{-1} u \in \varphi U \varphi^{-1} 
\end{align*}
which we just showed is impossible. 
\end{proof}

\begin{lemma}
If $u \in \varphi U \varphi^{-1}$ is non-trivial then 
\begin{align*}
E^+(\wedge^p u) \cap \Gr_p(\Rb^{2p}) \subset \Gr_{p}(\Rb^{2p}) \setminus \Mb.
\end{align*}
\end{lemma}

\begin{proof}
Suppose $u =  \begin{bmatrix} A & 0 \\ 0 & B \end{bmatrix}$. Then both $A,B$ are unipotent and 
\begin{align*}
u^m \begin{bmatrix} \Id_p \\ X \end{bmatrix} = \begin{bmatrix} \Id_p \\ B^m XA^{-m} \end{bmatrix}.
\end{align*}
Since both $B$ and $A$ are unipotent, for a generic $X \in M_{p,p}(\Rb)$ we have 
\begin{align*}
\lim_{m \rightarrow \infty} \norm{B^m X A^{-m} } = \infty.
\end{align*}
Which implies that $E^+(\wedge^p u) \cap \Gr_p(\Rb^{2p}) \subset \Gr_{p}(\Rb^{2p}) \setminus \Mb$. 
\end{proof}

Now we have a contradiction because 
\begin{align*}
e \in \Gr_p(\Rb^{2p}) \cap \overline{ \cup_{u \in \varphi U \varphi^{-1} } E^+(\wedge^p u) } \subset \Gr_{p}(\Rb^{2p}) \setminus \Mb
\end{align*}
and $e \in \Mb$.

\section{When $p=2$}\label{sec:p_2}

In this section we show that Cases (3) and (4) of Theorem~\ref{thm:reduction} are impossible.

\begin{theorem}\label{thm:p_2}
Suppose $\Omega \subset \Gr_2(\Rb^4)$ is a bounded convex open subset of some affine chart of $\Gr_2(\Rb^4)$ and there exists a discrete group $\Gamma \leq \Aut(\Omega)$ so that $\Gamma \backslash \Omega$ is compact. Then the connected component of the identity in $\Aut(\Omega)$ is a simple Lie group with trivial center that acts transitively on $\Omega$.
\end{theorem}

For the rest of the section let $\Omega \subset \Gr_2(\Rb^4)$ and $\Gamma \leq \Aut(\Omega)$ be as in the hypothesis of Theorem~\ref{thm:p_2}. As in Section~\ref{sec:initial}, let $G:=\Aut(\Omega)$ and let $G^0$ be the connected component of the identity of $G$. 

Define the subgroups 
\begin{align*}
G_1 :=  \left\{ \begin{bmatrix} A & 0 \\ 0 & A \end{bmatrix} : A \in \SL_2(\Rb) \right\}
\end{align*}
and 
\begin{align*}
G_2: = \left\{ \begin{bmatrix} a \Id_2 & b \Id_2 \\ c\Id_2 & d\Id_2 \end{bmatrix} : ad-bc = 1\right\}.
\end{align*}
By Theorem~\ref{thm:reduction} we may assume that either 
\begin{enumerate}
\item $G^0$ is a simple Lie group with trivial center that acts transitively on $\Omega$,
\item there exists a cocompact lattice $\Lambda \leq G_2$ so that $G_1 \times \Lambda$ has finite index in $\Aut(\Omega)$, 
\item $G_1 \times G_2$ has finite index in $\Aut(\Omega)$ and acts transitively on $\Omega$. 
\end{enumerate}

We rule out case (2) above by proving the following:

\begin{lemma}
With the notation above, $G^0$ has finite index in $\Aut(\Omega)$ and acts transitively on $\Omega$. 
\end{lemma}

\begin{proof} Suppose not, then by the remarks above there exists a cocompact lattice $\Lambda \leq G_2$ so that $G_1 \times \Lambda$ has finite index in $\Aut(\Omega)$. By possibly changing our cocompact lattice we may also assume that $\Gamma = \Gamma_1 \times \Lambda$ for some cocompact lattice $\Gamma_1 \leq G_1$. 

For a subgroup $H \leq \Aut(\Omega)$ let $\Lc(H)$ denote the set of points $x \in \partial \Omega$ where there exists some $y \in \Omega$ and sequence $h_n \in H$ so that $h_n y \rightarrow x$. Recall that $\Ext(\Omega) \subset \partial \Omega$ is the set of $\Rc$-extreme points of $\Omega$. Then define 
\begin{align*}
\Ext(H) : = \Lc(H) \cap \Ext(\Omega).
\end{align*} 

Let $e_1, \dots e_4$ be the standard basis of $\Rb^4$. Then a direct computation (using Part (4) of Theorem~\ref{thm:extreme}) shows that 
\begin{align*}
\Ext(G_1) = \{ [(\alpha e_1+ \beta e_2) \wedge (\alpha e_3 + \beta e_4)] : \alpha, \beta \in \Rb\}
\end{align*}
and  
\begin{align*}
\Ext(\Lambda) \subset \{ [(\alpha e_1+ \beta e_3) \wedge (\alpha e_2 + \beta e_4)] : \alpha, \beta \in \Rb\}.
\end{align*}
This description implies that $\overline{\Ext(G_1)}$ and $\overline{\Ext(\Lambda)}$ are disjoint and $\Gamma$-invariant sets. Moreover since $\Lambda \leq G_2$ is a cocompact lattice there exists some $\lambda \in \Lambda$ so that $\wedge^2 \lambda$ has a unique eigenvalue of maximum absolute value (see~\cite{P1994}). Then part (4) of Theorem~\ref{thm:extreme} implies that $\Ext(\Lambda) \neq \emptyset$. So suppose that $e \in \Ext(\Lambda)$. 

Now up to a projective isomorphism we can assume that $\Omega$ is a convex subset of the affine chart 
\begin{align*}
\Mb = \left\{ \begin{bmatrix} \Id_2 \\ X \end{bmatrix} : X \in M_{2,2}(\Rb) \right\}.
\end{align*}
and $ e=\begin{bmatrix} \Id_2 & 0\end{bmatrix}^t \in \partial \Omega$. Then by Corollary~\ref{cor:rescaling} there exists $\gamma_n \in \Gamma$ and $t_n \rightarrow \infty$ so that 
\begin{align*}
\varphi = \lim_{n \rightarrow \infty} \begin{bmatrix} \Id_2 & 0 \\ 0 & e^{t_n} \Id_2 \end{bmatrix} \gamma_n 
\end{align*}
exists in $\PGL_{4}(\Rb)$ and $\varphi(\Omega) = \mathcal{T}\mathcal{C}_e \Omega$. In particular, $\Omega$ is invariant under the one-parameter group 
\begin{align*}
\varphi^{-1} \left\{ \begin{bmatrix} \Id_p & 0 \\ 0 & e^t \Id_p \end{bmatrix} : t \in \Rb \right\}\varphi.
\end{align*}
This implies that $\varphi^{-1}(e) \in \Ext(G_1)$. But 
\begin{align*}
\gamma_n^{-1} \begin{bmatrix} \Id_2 & 0 \\ 0 & e^{-t_n} \Id_2 \end{bmatrix} e = \gamma_n^{-1} e \subset \Ext(\Lambda)
\end{align*}
and thus 
\begin{align*}
\varphi^{-1}(e) \in \Ext(G_1) \cap \overline{\Ext(\Lambda)}.
\end{align*}
This is a contradiction. 
\end{proof}

We rule out case (3) above by proving the following:

\begin{lemma}
$G^0$ is a simple Lie group. 
\end{lemma}

\begin{proof} Suppose not, then by the remarks above $G_1 \times G_2$ has finite index in $\Aut(\Omega)$. We may also assume that $\Gamma = \Gamma_1 \times \Gamma_2$ for some cocompact lattices $\Gamma_1 \leq G_1$ and $\Gamma_2 \leq G_2$.

Define the subgroups
\begin{align*}
K_1 = \left\{ \begin{bmatrix} A & 0 \\ 0 & A \end{bmatrix} : A \in \SO(2) \right\}
\end{align*}
and 
\begin{align*}
K_2=\left\{ \begin{bmatrix} a \Id_2 & b \Id_2 \\ c\Id_2 & d\Id_2 \end{bmatrix} : \begin{pmatrix} a & b \\ c & d \end{pmatrix} \in \SO(2)  \right\}.
\end{align*}
Then $K_1 \times K_2 \leq G_1 \times G_2$ is a maximal compact connected subgroup. Moreover the action of $K_1 \times K_2$ on $\Gr_2(\Rb^4)$ has no fixed points. 

Next let $K_x \leq \Aut(\Omega)$ be the connected component of the stabilizer of some $x \in \Omega$. Since $\Aut(\Omega)$ acts properly on $\Omega$ (see Proposition~\ref{prop:proper}), $K_x$ is a compact subgroup. Moreover, since $G^0 = G_1 \times G_2$, we see that $K_x \leq G_1 \times G_2$. Thus, since maximal compact subgroups are conjugate in semisimple Lie groups, there exists some $g \in G_1 \times G_2$ so that 
\begin{align*}
gK_x g^{-1} \leq K_1 \times K_2.
\end{align*}
But $\dim (K_1 \times K_2) =2$. Moreover
\begin{align*} 
6 - \dim(K_x) = \dim( G_1 \times G_2 / K_x) \leq \dim(\Omega) = 4
\end{align*}
so $\dim K_x \geq 2$. Thus $gK_x g^{-1} = K_1 \times K_2$. This contradicts the fact that $K_1 \times K_2$ has no fixed points in $\Gr_2(\Rb^4)$.
\end{proof} 

\section{Finishing the proof of Theorem~\ref{thm:main}}\label{sec:final_step}

Theorem~\ref{thm:reduction}, Theorem~\ref{thm:no_center}, Theorem~\ref{thm:no_unipotent}, and Theorem~\ref{thm:p_2} reduce the proof of Theorem~\ref{thm:main} to the following:

\begin{theorem}\label{thm:final}
Suppose $p > 1$ and $\Omega \subset \Gr_p(\Rb^{2p})$ is a bounded convex open subset of some affine chart of $\Gr_p(\Rb^{2p})$.
If the connected component of the identity of $\Aut(\Omega)$ is a simple Lie group with trivial center which acts transitively on $\Omega$, then $\Omega$ is projectively isomorphic to $\Bc_{p,p}$. 
\end{theorem}

For the rest of the section suppose that $\Omega$ satisfies the hypothesis of Theorem~\ref{thm:final}. As in Section~\ref{sec:initial}, let $G:=\Aut(\Omega)$ and let $G^0$ be the connected component of the identity of $G$. Also let $e_1, \dots, e_{2p} \in \Rb^{2p}$ be the standard basis. 

Throughout the argument we will replace $\Omega$ with translates $g\Omega$ for some $g \in \PGL_{2p}(\Rb)$.  This will have the effect of replacing $G$ with $gGg^{-1}$. 

Fix some $x_0 \in \Omega$ and let $K \leq G^0$ be the connected component of the stabilizer of $x_0$. By Remark~\ref{rmk:max_cpct}, $K$ is a finite index subgroup of some maximal compact subgroup of $G^0$. Moreover, since $K$ is compact, by translating $\Omega$ we may assume that $K \leq \PSO(2p)$. Then since $\PSO(2p)$ acts transitively on $\Gr_{p}(\Rb^{2p})$ we can translate $\Omega$ and assume that $x_0 = [e_1 \wedge \dots \wedge e_p]$. Then 
\begin{align*}
K \leq \left\{ \begin{bmatrix} A & 0 \\ 0 & B \end{bmatrix} : A, B \in \SO(p)\right\}.
\end{align*}
In particular, $\dim(K) \leq p(p-1)$. 

\begin{lemma}
With the notation above, $\rank(G^0) \geq p$.
\end{lemma}

\begin{proof} Using the Cartan decomposition there exists a connected abelian group $A \leq G^0$ so that $\dim(A) = \rank(G^0)$ and $KAK = G^0$. In particular, in the matrix model of $\Gr_p(\Rb^{2p})$ 
\begin{align*}
\Omega = KAK \cdot \begin{bmatrix} \Id_p & 0 \end{bmatrix}^t = K A \cdot \begin{bmatrix} \Id_p & 0 \end{bmatrix}^t.
\end{align*}
Thus we must have 
\begin{equation}
\label{eq:rank}
\dim(K) + \dim(A) \geq \dim(\Omega) = p^2.
\end{equation}
Since $\dim(K) \leq p(p-1)$ we then have 
\begin{align*}
\rank(G^0) =\dim(A) \geq p.
\end{align*}
\end{proof}

\begin{lemma}
With the notation above, $G^0$ is isomorphic to $\PSO(p,p)$. 
\end{lemma}

\begin{proof}
Now 
\begin{align*}
\dim \left( G^0/K \right) = \dim(\Omega) = p^2
\end{align*}
and 
\begin{align*}
\rank(G^0) \geq p.
\end{align*}
In particular
\begin{align*}
\rank(G^0) \geq \sqrt{ \dim \left( G^0/K \right)}.
\end{align*}
The only two simple Lie groups of non-compact type and with trivial center with this property are $\PSL_{d+1}(\Rb)$ for $d \geq 3$ and $\PSO(d,d)$ for $d\geq2$ (see the classification of simple Lie groups in~\cite[Chapter X]{H1978}). 

If $G^0$ is isomorphic to $\PSL_{d+1}(\Rb)$ then $K$ is isomorphic to $\PSO(d+1)$. In particular $K$ is a simple Lie group and
\begin{align*}
\dim K = \frac{(d+1)(d)}{2}
\end{align*}
Next consider the natural projections, 
\begin{align*}
\pi_1, \pi_2 : K \rightarrow \PSO(p)
\end{align*}
given by 
\begin{align*}
\pi_1 \left( \begin{bmatrix} A & 0 \\ 0 & B \end{bmatrix} \right) = A \text{ and } \pi_2 \left( \begin{bmatrix} A & 0 \\ 0 & B \end{bmatrix} \right) = B
\end{align*}
Now since $K$ is simple either $(\pi_1 \times \pi_2):K \rightarrow \PSO(p) \times \PSO(p)$ is trivial or injective. But
\begin{align*}
\ker (\pi_1 \times \pi_2) \leq \left\{ \Id_{2p}, \begin{bmatrix} \Id_p & \\ & -\Id_p \end{bmatrix},\begin{bmatrix} -\Id_p & \\ & \Id_p \end{bmatrix} \right\}.
\end{align*}
so $\pi_1 \times \pi_2$ is injective. Thus at least one $\pi_i$ has non-trivial image. Then by the simplicity of $K$ we see that $K \cong \pi_i(K) \leq \PSO(p)$. So 
\begin{align*}
\dim K \leq \frac{p(p-1)}{2}
\end{align*}
and so
\begin{align*}
 (d+1)(d) \leq p(p-1).
 \end{align*}
 Thus $d \leq p+1$. But then we have a contradiction, because by Equation~\ref{eq:rank}, we have
 \begin{align*}
 p^2 \leq \rank(G^0) + \dim(K) \leq d + \frac{p(p-1)}{2} \leq p+1  + \frac{p(p-1)}{2} = \frac{p^2+p}{2}+1
\end{align*}
which is only true when $p=2$. Then $d=p+1 = 3$, but 
\begin{align*}
\dim \PSL_{4}(\Rb) / \PSO(4) = 9 \neq 4 = \dim \Omega
\end{align*}
so this case is impossible. 

Thus we must have that $G^0$ is isomorphic to $\PSO(p,p)$. 
\end{proof}

Now the inclusion $G^0 \leq \PGL_{2p}(\Rb)$ induces a representation $\phi: \PSO(p,p) \rightarrow \PGL_{2p}(\Rb)$. Notice that replacing $\Omega$ with $g\Omega$ for some $g \in \PGL_{2p}(\Rb)$ has the effect of replacing $\phi$ with $\Ad(g) \circ \phi$. 

At this point there is a number of ways to deduce that this representation is conjugate to the standard inclusion, but we will use the representation theory of $\SO(2p,\Cb)$ because it is appears explicitly in standard references (for instance~\cite{FH1991}).

Now since $K$ has finite index in a maximal compact subgroup of $G^0 \cong \PSO(p,p)$ and 
\begin{align*}
K \leq \left\{ \begin{bmatrix} A & 0 \\ 0 & B \end{bmatrix} : A, B \in \SO(p)\right\}.
\end{align*}
so we see that
\begin{align*}
K = \left\{ \begin{bmatrix} A & 0 \\ 0 & B \end{bmatrix} : A, B \in \SO(p)\right\}.
\end{align*}
Then since maximal compact subgroups are conjugate in $G^0$ we may translate $\Omega$ to assume that 
\begin{align*}
\phi( P(\SO(p) \times \SO(p))) = P( \SO(p) \times \SO(p)).
\end{align*}
Now if $K_1 = P(\SO(p) \times \{\Id_p\})$ and $K_2 = P(\SO(p) \times \{\Id_p\})$ then, using the simplicity of $K_1,K_2$ and the fact that $\phi(K_1), \phi(K_2)$ commute, we see that 
\begin{align*}
\{ \phi(K_1), \phi(K_2)\} = \{K_1, K_2\}.
\end{align*}
So by translating $\Omega$ we may assume that $\phi(K_1) = K_1$ and $\phi(K_2) = K_2$. Now each $K_i$ is isomorphic to $\SO(p)$. 

\begin{lemma}
If $f:\SO(p) \rightarrow \SO(p)$ is an automorphism then there exists some $h \in \mathrm{O}(p)$ so that $f(k) = h k h^{-1}$ for all $k \in \SO(p)$. So we can translate $\Omega$ and assume that $\phi(k)=k$ for all $k \in K_1 \cup K_2$. 
\end{lemma}

\begin{proof}
Suppose $f:\SO(p) \rightarrow \SO(p)$ is an automorphism. Then the induced map on Lie algebras $d(f): \Lso(p) \rightarrow \Lso(p)$ is an automorphism. By~\cite[Proposition D.40]{FH1991} there is a group isomorphism between $\Aut(\Lso(p))/\Inn(\Lso(p))$ and the automorphisms of the Dynkin diagram of $\Lso(p)$. 

When $p=2n+1$, the Dynkin diagram is $B_n$ which has trivial automorphism group. Thus when $p$ is odd every automorphism is inner. 

When $p=2n$ the Dynkin diagram is $D_n$. Now $D_n$ has a non-trivial automorphism which is induced by the map
\begin{align*}
k \rightarrow hkh^{-1}
\end{align*}
where $h \in \mathrm{O}(p)$ and $\det h=-1$. When $n \neq 4$, the automorphism group of $D_n$ is $\Zb/2\Zb$ and so this is the only non-trivial automorphism. When $n=4$, the automorphism group is of $D_n$ is isomorphic to the symmetric group on $3$ elements, however it is well known that $\Aut(\SO(8))$ only induces two of them (see for instance~\cite[Section 20.3]{FH1991}).

\end{proof}

Now let $d(\phi): \Lso(p,p) \rightarrow \Lsl_{2p}(\Rb)$ be the corresponding Lie algebra representation. We can complexity to obtain a representation $d(\phi): \Lso(2p,\Cb) \rightarrow \Lsl_{2p}(\Cb)$. But then by the classification of irreducible representations of $\SO(2p,\Cb)$ (see for instance~\cite[Chapter 19]{FH1991}) we see that there exists $g \in \SL_{2p}(\Cb)$ so that 
\begin{align*}
\Ad(g) d(\phi)  = \iota
\end{align*}
where $\iota: \Lso(2p,\Cb) \hookrightarrow \Lsl_{2p}(\Cb)$ is the standard inclusion representation. Since 
\begin{align*}
g^{-1}\begin{pmatrix} X_1 & 0 \\ 0 & X_2 \end{pmatrix} g=d(\phi)\begin{pmatrix} X_1 & 0 \\ 0 & X_2 \end{pmatrix} =  \begin{pmatrix} X_1 & 0 \\ 0 & X_2 \end{pmatrix} 
\end{align*}
for all $X_1,X_2 \in \Lso(p)$ it is easy to see that 
\begin{align*}
g = \begin{pmatrix} \alpha \Id_p & 0 \\ 0 & \alpha^{-1} \Id_p \end{pmatrix}
\end{align*}
for some $\alpha \in \Cb^*$. Now 
\begin{align*}
g \begin{pmatrix} A & B \\ C & D \end{pmatrix} g^{-1} = \begin{pmatrix} A & \alpha^2 B \\ \alpha^{-2} C & D \end{pmatrix} 
\end{align*}
and $g d(\phi)(\Lso(p,p))g^{-1} = \Lso(p,p)$. So $\alpha^2 \in \Rb$. Thus $\alpha = \lambda i$ for some $\lambda \in \Rb^*$. But then if 
\begin{align*}
g_0 = -i g = \begin{pmatrix} \lambda \Id_p & 0 \\ 0 & -\lambda^{-1} \Id_p \end{pmatrix}
\end{align*}
we see that 
\begin{align*}
\Ad(g_0) d(\phi) = \Ad(g) d(\phi) = \iota.
\end{align*}
So if we replace $\Omega$ by $g_0 \Omega$ we see that  $\phi: \PSO(p,p) \hookrightarrow \PGL_{2p}(\Rb)$ is the standard inclusion representation and so $G^0 = \PSO(p,p)$. 

Finally 
\begin{align*}
\Omega = G^0 \cdot x_0 = \PSO(p,p) \cdot \begin{bmatrix} \Id_p \\ 0 \end{bmatrix} = \Bc_{p,p}
\end{align*}
and the theorem is proven.

\part{Appendices}

\appendix

\section{Proof of Theorem~\ref{thm:completeness}}\label{sec:proof_complete}

In this section we prove that (1) implies (2) in Theorem~\ref{thm:completeness}:

\begin{theorem}
Suppose $\Mb \subset \Gr_p(\Rb^{p+q})$ is an affine chart and $\Omega \subset \Mb$ is an open convex set. If $\Omega$ is $\Rc$-proper then $K_{\Omega}$ is a complete length metric on $\Omega$.
\end{theorem}

We will use some basic properties of the Hilbert metric $H_{\Cc}$ on a convex set $\Cc \subset \Rb^d$. In particular we will use:
\begin{enumerate}
\item (equivariance) If $A \in \Aff(\Rb^d)$ then $H_{A \Cc}(A x, A y) = H_{\Cc}(x,y)$, 
\item (properness) If $x \in \partial \Cc$ and $x_n \in \Cc$ is a sequence with $x_n \rightarrow x$ then 
\begin{align*}
H_{\Cc}(x_0, x_n) \rightarrow \infty,
\end{align*}
\item (completeness) If $\Cc$ contains no affine lines then $H_{\Cc}$ is a complete metric, 
\item If $\Cc = \Rb^d \times \Cc^\prime$ then 
\begin{align*}
H_{\Cc}((x_1, y_1), (x_2, y_2)) = H_{\Cc^\prime}(y_1, y_2).
\end{align*}
\end{enumerate}
All these properties follow immediately from the cross ratio definition of the Hilbert metric.

\begin{proof}
Identify $\Mb$ with the set of $q$-by-$p$ matrices and let $\Mb_1 \subset \Mb$ be the subset of rank one matrices. Define a function $\delta_{\Omega}: \Omega \times \Mb_1 \rightarrow \Rb_{\geq 0}$ by 
\begin{align*}
\delta_{\Omega}(x;v) = \inf \{ \norm{y-x} : y \in \partial \Omega \cap (x+\Rb v) \}.
\end{align*}
Since $\Omega$ is $\Rc$-proper, we must have that $\delta_{\Omega}(x;v) < \infty$ for all $x \in \Omega$ and $v \in \Mb_1$. Moreover, since $\Omega$ is convex, $\delta_{\Omega}$ is a continuous function. 

We will first show that $K_{\Omega}$ is a metric, using Proposition~\ref{prop:basic_d} we only need show that $K_{\Omega}(x,y) > 0$ for $x,y \in \Omega$ distinct. Now we can find $\epsilon >0$ such that the closed Euclidean ball 
\begin{align*}
B_{\epsilon}(x) = \{ z \in \Mb : \norm{x-z} \leq \epsilon \} 
\end{align*} 
is contained in $\Omega$ but $y \notin B_{\epsilon}(x)$. Since $\delta_{\Omega}$ is continuous, there exists $M >0$ such that 
\begin{align*}
\delta_{\Omega}(z;v) \leq M
\end{align*}
 for all $z \in B_{\epsilon}(x)$ and $v \in \Mb_1$. 
 
We claim that if $[z_1, z_2] \subset B_{\epsilon}(x)$ then $\rho_{\Omega}(z_1, z_2) \geq 1/(\epsilon+M) \norm{z_1-z_2}$. If $z_2-z_1 \notin \Mb_1$ then $\rho_\Omega(z_1, z_2) = \infty$ so we may assume that $z_2 - z_1 \in \Mb_1$. Then let $(a,b) = \overline{z_1 z_2} \cap \Omega$ labelled so that $a, z_1, z_2, b$ is the ordering along the line. By relabeling we may assume that $\norm{a-z_1} = \delta_{\Omega}(z_1, z_1-z_2) \leq M$. Then 
 \begin{align*}
 \rho_{\Omega}(z_1, z_2)
 & =  \abs{ \log\frac{ \norm{z_1-a}\norm{z_2-b}}{\norm{z_1-b}\norm{z_2-a}}} \geq \log \frac{ \norm{z_2-a}}{\norm{z_1-a}}  \\
 & = \int_{\norm{z_1-a}}^{\norm{z_2-a}} \frac{dt}{t} \geq \frac{1}{\norm{z_2-a}} \left( \norm{z_2-a} - \norm{z_1-a} \right).
 \end{align*}
 Since $z_1, z_2, a$ are all collinear and $\norm{z_1 - z_2} \leq \epsilon$ we then have 
  \begin{align*}
 \rho_{\Omega}(z_1, z_2) \geq \frac{1}{M+\epsilon}\norm{z_1-z_2}.
 \end{align*}
 Now we wish to show that $K_{\Omega}(x,y) > 0$. We claim that 
 \begin{align*}
  \rho_{\Omega}(x,a_1) + \sum_{i=1}^{n-1} \rho_{\Omega}(a_i,a_{i+1}) + \rho_{\Omega}(a_n,y) \geq \frac{\epsilon}{M+\epsilon}.
 \end{align*}
 for any $a_1, \dots, a_n \in \Omega$. This will imply that $d_{\Omega}(x,y) > 0$. Now by definition if $a,b \in \Mb$ and $c\in [a,b]$  then 
 \begin{align*}
 \rho_\Omega(a,b)+\rho_\Omega(b,c) = \rho_\Omega(a,c).
 \end{align*}
So without loss of generality there exists $1 \leq \ell < n$ such that $a_1, \dots, a_\ell \in B_{\epsilon}(x)$ and $a_{\ell+1} \in \partial B_{\epsilon}(x)$. Then by the above calculation 
 \begin{align*}
   \rho_{\Omega}(x,a_1) + \sum_{i=1}^{\ell} \rho_{\Omega}(a_i,a_{i+1}) \geq \frac{1}{M+\epsilon} \left( \norm{x-a_1} + \sum_{i=1}^{\ell} \norm{a_i-a_{i+1}} \right) \geq  \frac{\epsilon}{M+\epsilon}.
 \end{align*}
  This shows that $K_{\Omega}$ is a metric. 
 
 We will next show that $K_{\Omega}$ is a length metric. This follows from the fact that if $x, y \in \Omega$ and $x-y \in \Mb_1$ then 
 \begin{align*}
 \rho_{\Omega}(x,y) = \rho_{\Omega}(x,z) + \rho_{\Omega}(z,y)
 \end{align*}
 for any $z\in [x,y]$. Thus when $x-y \in \Mb_1$, there is a curve of length at most $\rho_{\Omega}(x,y)$ joining $x$ to $y$. Then by definition for any $x,y \in \Omega$ there exists a sequence of curves $\sigma_n$ joining $x$ to $y$ and whose length converges to $K_{\Omega}(x,y)$. 
 
Next we show that $K_{\Omega}$ is proper, that is for any $x_0 \in \Omega$ and $R \geq 0$ the closed metric balls $B=\{ x \in \Omega : K_{\Omega}(x,x_0) \leq R\}$ are compact. Let $x_n \in B$ be a subsequence, we will show that a subsequence of $x_n$ converges in $B$. By passing to a subsequence we can suppose that $x_n \rightarrow x \in \Mb$ or $x_n \rightarrow \infty$. 

If $x_n \rightarrow x \in \Mb$ and $x \in \Omega$ then $x \in B$ by part (5) of Proposition~\ref{prop:basic_d}. Otherwise $x \in \partial \Omega$. Let $H_\Omega$ be the Hilbert metric on $\Omega$, then $H_\Omega \leq K_\Omega$ by Kobayashi's construction of the Hilbert metric (described in Subsection~\ref{subsec:hilbert}).  So
\begin{align*}
K_{\Omega}(x_0,x_n) \geq H_{\Omega}(x_0,x_n) \rightarrow \infty
\end{align*}
which is a contradiction. 

Finally suppose that $x_n \rightarrow \infty$. If $\Omega$ contains no affine lines then $H_{\Omega}$ is a proper metric and so
\begin{align*}
K_{\Omega}(x_0,x_n) \geq H_{\Omega}(x_0,x_n) \rightarrow \infty.
\end{align*}
If $\Omega$ is not proper, then we can identify $\Mb$ with $\Rb^{D}$ where $D = pq$ and find an affine map $\Phi \in \Aff(\Rb^D)$ so that $\Phi \Omega = \Rb^d \times \Omega^\prime$ where $\Omega^\prime$ is a proper convex set and $d \leq D$. Notice that $H_{\Omega}(z_1,z_2) = H_{\Phi \Omega}(\Phi z_1, \Phi z_2)$ for all $z_1, z_2 \in \Omega$ but the metric $K_{\Phi \Omega}$ and $K_{\Omega}$ have no clear relationship because $\Phi$ will in general not preserve the rank one lines. Since $\Omega$ is $\Rc$-proper we must have that $d < D$. Let $\pi: \Rb^D \rightarrow \Rb^{D-d}$ be the projection onto the second factor. Next let $\sigma_n:[0,1] \rightarrow \Omega$ be a curve joining $x_0$ to $x_n$ with $K_\Omega$-length less than $R+\epsilon$. 

We claim that the set $\overline{\{\pi(\Phi\sigma_n(t)) : n \in \Nb, t \in [0,1] \}}$ is a compact subset of $\Omega^\prime$. This follows from the fact that 
\begin{align*}
R+\epsilon \geq K_{\Omega}(x_0,\sigma_n(t))&  \geq H_{\Omega}(x_0,\sigma_n(t)) = H_{\Phi \Omega}(\Phi x_0, \Phi \sigma_n(t)) \\
& = H_{\Omega^\prime}(\pi (\Phi x_0),\pi(\Phi \sigma_n(t)))
\end{align*}   
and the fact that $H_{\Omega^\prime}$ is a proper metric on $\Omega^\prime$. So if $x_n = \Phi^{-1}(y_n, z_n)$, we must have $y_n \rightarrow \infty$. But then notice that 
\begin{align*}
\delta_{\Omega}(x+a; v) =  \delta_{\Omega}(x; v) 
\end{align*}
for all $a \in \Phi^{-1}(\Rb^d \times \{0\})$ and $v \in \Mb_1$. And so there exists $M \geq 0$ such that 
\begin{align*}
\delta_{\Omega}(\sigma_n(t);v) \leq M
\end{align*}
for all $n \in \Nb$, all $t \in [0,1]$, and $v \in \Mb_1$. But then arguing as before we see that
\begin{align*}
\length(\sigma_n) \geq \frac{1}{M} \norm{x_0-x_n}.
\end{align*}
Since $x_n \rightarrow \infty$ and $\length(\sigma_n) < R+\epsilon$ we have a contradiction. 

Finally we observe that $K_{\Omega}$ is a complete metric on $\Omega$. If $(x_n)$ is a Cauchy sequence then we can pass to a subsequence so that 
\begin{align*}
\sum_{n=1}^{\infty} \norm{x_n-x_{n+1}} = R < \infty.
\end{align*}
But then $x_n \in \{ x \in \Omega: d_{\Omega}(x_1,x) \leq R\}$ which is a compact subset of $\Omega$. 
\end{proof}

\section{Proof of Theorem~\ref{thm:haus_conv}}\label{sec:proof_of_haus}

In this section we prove Theorem~\ref{thm:haus_conv}:

\begin{theorem}\label{thm:haus_conv_app}
Let $\Mb$ be an affine chart of $\Gr_p(\Rb^{p+q})$ and suppose $\Omega_n \subset \Mb$ is a sequence of $\Rc$-proper convex open sets converging to a $\Rc$-proper convex open set $\Omega \subset \Mb$ in the local Hausdorff topology. Then 
\begin{align*}
K_{\Omega}(x,y) = \lim_{n \rightarrow \infty} K_{\Omega_n}(x,y)
\end{align*}
for all $x,y \in \Omega$ uniformly on compact sets of $\Omega \times \Omega$.
\end{theorem}

\begin{lemma} With the notation above, for any compact subset $K \subset \Omega$ and $\epsilon > 0$ there exists $N >0$ so that
\begin{align*}
(1-\epsilon) \rho_{\Omega_n}(x,y) \leq \rho_{\Omega}(x,y) \leq (1+\epsilon) \rho_{\Omega_n}(x,y)
\end{align*}
for all $x,y \in K$ and $n \geq N$. 
\end{lemma}

\begin{proof} Fix $K \subset \Omega$ compact and $\epsilon > 0$. Let 
\begin{align*}
C = \{ (x,y) \in K \times K : \dimension(x \cap y) \geq p-1\} \subset K \times K.
\end{align*}
Then $C$ is exactly the pairs of points in $K$ where $\rho_\Omega$ is finite. Since $\Omega$ is convex the function $\rho_{\Omega}|_C$ is continuous. Now suppose, for a contradiction that there exists $n_k \rightarrow \infty$ so that $x_k \neq y_k$, $(x_k, y_k) \in C$, and
\begin{align*}
\frac{\rho_{\Omega}(x_k,y_k)}{\rho_{\Omega_{n_k}}(x_k,y_k)} \notin [1-\epsilon, 1+\epsilon].
\end{align*}
By passing to a subsequence we can suppose that $x_k \rightarrow x$ and $y_k \rightarrow y$. Let 
\begin{align*}
\{ a_k,b_k\} = \partial \Omega_{n_k} \cap \overline{x_k y_k}
\end{align*}
ordered so that 
\begin{align*}
 \rho_{\Omega_{n_k}}(x_k, y_k) = \log \frac{ \abs{x_k-b_k}\abs{y_k-a_k}}{\abs{x_k-a_k}\abs{y_k-b_k}}.
 \end{align*}
 By passing to another subsequence we can suppose that $a_k \rightarrow a$ and $b_k \rightarrow b$. Now since $\Omega_n$ converges to $\Omega$ in the local Hausdorff topology we see that $a,b \in \partial \Omega \cap \overline{xy}$. In particular, 
 \begin{align*}
 \rho_{\Omega}(x, y) = \log \frac{ \abs{x-b}\abs{y-a}}{\abs{x-a}\abs{y-b}} = \lim_{k \rightarrow \infty} 
\log \frac{ \abs{x_k-b_k}\abs{y_k-a_k}}{\abs{x_k-a_k}\abs{y_k-b_k}} = \lim_{k \rightarrow \infty} \rho_{\Omega_{n_k}}(x_k, y_k).
 \end{align*}
 However, by continuity of $\rho_\Omega$ on $C$
  \begin{align*}
 \rho_{\Omega}(x, y) = \lim_{k \rightarrow \infty} \rho_{\Omega}(x_k, y_k)
 \end{align*}
which is a contradiction. 
\end{proof}

\begin{proof}[Proof of Theorem~\ref{thm:haus_conv_app}]
Now suppose that $K \subset \Omega$ is compact. Then we can pick $R > 0$ and $x_0 \in \Omega$ so that $K \subset \{ x \in \Omega: d_{\Omega}(x,x_0) \leq R\}$. Let 
\begin{align*}
K^\prime = \left\{ x \in \Omega: K_{\Omega}(x,x_0) \leq (1+\epsilon)^2 (R+1)+R +\epsilon \right\}.
\end{align*}
Next pick $N >0$ so that 
\begin{align*}
(1-\epsilon) \rho_{\Omega_n}(x,y) \leq \rho_{\Omega}(x,y) \leq (1+\epsilon) \rho_{\Omega_n}(x,y)
\end{align*}
for all $x,y \in K^\prime$ and $n \geq N$. Now we claim that 
\begin{align*}
K_{\Omega_n}(x,y) \leq (1+\epsilon) K_{\Omega}(x,y)
\end{align*}
for $x,y \in K$ and $n \geq N$. For $x,y \in K$ and $\delta \in (0,1)$ pick $x=a_0, a_1, \dots, a_m=y$ so that 
\begin{align*}
 \rho_{\Omega}(x,a_1) + \rho_{\Omega}(a_1,a_2) + \dots + \rho_{\Omega}(a_{m-1}, y) \leq K_{\Omega}(x,y) + \delta.
 \end{align*}
 Then $a_0, \dots, a_m \in K^\prime$ and so 
 \begin{align*}
 \rho_{\Omega_n}(x,a_1) + \rho_{\Omega_n}(a_1,a_2) + \dots + \rho_{\Omega_n}(a_{m-1}, y) \leq (1+\epsilon)(K_{\Omega}(x,y) + \delta)
 \end{align*}
for $n \geq N$. Since $\delta>0$ was arbitrary we see that 
\begin{align*}
K_{\Omega_n}(x,y) \leq (1+\epsilon) K_{\Omega}(x,y)
\end{align*}
for $x,y \in K$ and $n \geq N$. 

Now suppose $n \geq N$,  $x,y \in K$, $\delta \in (0,1)$, and $x=a_0, a_1, \dots, a_m =y \in \Omega_n$ so that
\begin{align*}
 \rho_{\Omega_n}(x,a_1) + \rho_{\Omega_n}(a_1,a_2) + \dots + \rho_{\Omega_n}(a_{m-1}, y) \leq K_{\Omega_n}(x,y) + \delta.
 \end{align*}
 If $a_0, a_1, \dots, a_m  \in K^\prime$ then we immediately see that 
 \begin{align*}
 K_{\Omega}(x,y) \leq \rho_{\Omega}(x,a_1) + \rho_{\Omega}(a_1,a_2) + \dots + \rho_{\Omega}(a_{m-1}, y) \leq (1+\epsilon) (K_{\Omega_n}(x,y) + \delta)
 \end{align*}
 Otherwise we can assume that there is some $a_\ell$ so that $a_{\ell} \in \partial K^\prime$. Then $K_{\Omega}(a_\ell, x_0) = (1+\epsilon)^2 (R+1)+R +\epsilon$ and so 
  \begin{align*}
(1+\epsilon)^2 (R+1) +\epsilon & \leq K_{\Omega}(x_0,a_{\ell})-K_{\Omega}(x_0,x) \leq K_{\Omega}(x,a_{\ell})\\
&  \leq \rho_{\Omega}(x,a_1) + \rho_{\Omega}(a_1,a_2) + \dots + \rho_{\Omega}(a_{\ell-1}, a_\ell) \\
& \leq (1+\epsilon) (K_{\Omega_n}(x,y) + \delta) \leq (1+\epsilon)\Big( (1+\epsilon)K_{\Omega}(x,y) + 1 \Big)\\
& \leq (1+\epsilon)^2 (R+1)
 \end{align*}
 which is a contradiction. Thus $a_0, a_1, \dots, a_m  \in K^\prime$ and 
 \begin{align*}
 K_{\Omega}(x,y) \leq (1+\epsilon) (K_{\Omega_n}(x,y) + \delta)
 \end{align*}
 Since $\delta \in (0,1)$ was arbitrary we see that 
 \begin{align*}
 K_{\Omega}(x,y) \leq (1+\epsilon) K_{\Omega_n}(x,y).
 \end{align*}
\end{proof}

\bibliographystyle{alpha}
\bibliography{hilbert}

\end{document}